\documentclass[12pt]{amsart}
\usepackage{fullpage,tikz,scalerel}
\usepackage[utf8]{inputenc}
\usepackage{color,amsfonts,amsmath,amssymb,amsthm,cases,wasysym}
\usepackage[normalem]{ulem}
\usepackage{svg}
\usepackage{enumitem}
\usepackage{booktabs}
\usepackage{mathtools}
\usepackage{hyperref} 

\title{Framing Triangulations and Framing Posets of Planar DAGs with Nontrivial Netflow Vectors}
\author{Jonah Berggren}
\address{Department of Mathematics, University of Kentucky, Lexington, KY, United States}
\email{jrberggren@uky.edu}

\newtheorem{thm}{Theorem}[section]
\newtheorem{prop}[thm]{Proposition}
\newtheorem{lemma}[thm]{Lemma}
\newtheorem{cor}[thm]{Corollary}

\theoremstyle{definition}
\newtheorem{defn}[thm]{Definition}
\newtheorem{remk}[thm]{Remark}
\newtheorem{example}[thm]{Example}
\newtheorem{question}[thm]{Question}
\newtheorem{conjecture}[thm]{Conjecture}

\newtheorem{thmIntro}{Theorem}    
    
\newtheorem{exIntro}[thmIntro]{Example}

\newcommand{\ucong}{\overset{\textup{int}}{\cong}}
\newcommand{\F}{\mathcal F} 
\newcommand{\G}{\Gamma} 
\renewcommand{\a}{\mathbf{a}} 
\newcommand{\p}{\mathbf{p}} 
\newcommand{\q}{\mathbf{q}} 
\renewcommand{\r}{\mathbf{r}} 
\newcommand{\I}{\mathcal{I}} 
\newcommand{\J}{\mathcal{J}} 
\newcommand{\K}{K} 
\renewcommand{\L}{L} 
\newcommand{\inn}{\text{in}} 
\newcommand{\out}{\text{out}} 
\renewcommand{\P}{P} 
\newcommand{\R}{\mathfrak F} 

\newcommand{\up}{\uparrow}
\newcommand{\down}{\downarrow}

\newcommand{\updown}{\updownarrow}

\newcommand{\routes}{\text{Routes}}

\newcommand{\D}{\Delta}

\usetikzlibrary {arrows.meta, decorations.markings}
\newcommand{\shuffle}{\scalerel*{\tikz[xscale=-1]{\useasboundingbox(-.2,-.4)rectangle(.2,.3);
    \draw [line width=1.5, -{[length=2]}, line cap=round, domain=-0.2:0.2 ,variable=\t, smooth, samples=10] plot (\t, {sqrt(0.2*0.2-\t*\t)-0.27});
    \draw [line width=1.5, -{[length=2]}, line cap=round, domain=-0.2:0.2 ,variable=\t, smooth, samples=10] plot (\t, {-sqrt(0.2*0.2-\t*\t)+0.27});}}{(}
}
\newcommand{\realignment}{\scalerel*{\tikz[xscale=-1]{\useasboundingbox(-.2,-.4)rectangle(.2,.3);
    \draw [line width=1.5, -{[length=2]}, line cap=round, domain=-0.2:0.2 ,variable=\t, smooth, samples=10] plot (\t, {-\t*\t/0.2*1.8+.12});
    \draw [line width=1.5, -{[length=2]}, line cap=round, domain=-0.2:0.2 ,variable=\t, smooth, samples=10] plot (\t, {\t*\t/0.2*1.8-.12});}}{(}
}
\newcommand{\rotation}{\scalerel*{\tikz[xscale=-1]{\useasboundingbox(-.2,-.4)rectangle(.2,.3);
    \draw [line width=1.5, -{[length=2]}, line cap=round, domain=-0.2:0.2 ,variable=\t, smooth, samples=10] plot (\t, {\t});
    \draw [line width=1.5, -{[length=2]}, line cap=round, domain=-0.2:0.2 ,variable=\t, smooth, samples=10] plot (\t, {-\t});}}{(}
}

\newcommand{\sh}{\shuffle}
\newcommand{\re}{\realignment}
\newcommand{\ro}{\rotation}

\begin{document}

\maketitle

\begin{abstract}
	The polyhedron of unit flows on a directed acyclic graph (DAG) with one source and one sink is known to admit regular unimodular triangulations induced by framings of the DAG. 
	The dual graph of any of these triangulations may be given the structure of the Hasse diagram of a lattice, generalizing many variations of the Tamari lattice and the weak order. 
	We extend this theory to flow polytopes of DAGs which may have multiple sources, multiple sinks, and nontrivial netflow vectors under certain planarity conditions. We construct a unimodular triangulation of such a flow polytope indexed by combinatorial data and give a poset structure on its maximal simplices generalizing the strongly planar one-source-one-sink case.
\end{abstract}

\setcounter{tocdepth}{1}
\tableofcontents

\section{Introduction}

The polyhedra of unit \emph{flows} on a directed acyclic graph (DAG) are a fundamental object of study in combinatorial optimization with relations to many fields of mathematics~\cite{RT,DH,GP,TG}.
In the field of combinatorics, flow polytopes attract much attention due in part to techniques for calculating volumes~\cite{LIDSKII} and triangulations.

Danilov, Karzanov, and Koshevoy~\cite{DKK} introduced \emph{framings} on DAGs and used them to induce regular unimodular \emph{framing triangulations} on the associated unit flow polytopes (i.e., flows with netflow vector $(1,0,\dots,0,-1)$).
The dual graphs of these framing triangulations were shown to have the structure of the Hasse diagram of a lattice~\cite{vBC} called the \emph{framing lattice} of the framed DAG.
Many interesting classes of lattices appear as framing lattices of unit flow polytopes of DAGs with one source and one sink -- notably,
\begin{itemize}
	\item Tamari lattices and related orders (for example, $\nu$-Tamari lattices~\cite{vBGMY}, grid-Tamari orders~\cite{MGrid}, type A Cambrian $\varepsilon$-lattices~\cite{Reading}, alt $\nu$-Tamari lattices~\cite{CC}), and
	\item the weak order and related orders (for example, $s$-weak orders~\cite{GMPTY} and multipermutation lattices~\cite{BB}).
\end{itemize}
See~\cite{vBC} for more details on the appearance of these and other lattices as framing lattices.
Framing triangulations and framing lattices have shed light on $h^*$-vectors of certain classes of flow polytopes~\cite[Lemma 6.10, Corollary 6.11, Corollary 6.20]{WIWT} and established connections to volume formulas for flow polytopes, order polytopes~\cite{MMS}, cluster algebras~\cite[\S8]{DKK}, and the representation theory of gentle algebras~\cite{WIWT,BS,BERG}.

Thus far, this story has been limited to DAGs with a unique source, a unique sink, and unit netflow vector $(1,0,\dots,0,-1)$. The aim of this paper is to obtain unimodular framing triangulations and framing posets for flow polytopes from DAGs which may have multiple sources, multiple sinks, and nontrivial netflow vectors.
As we will discuss later in the introduction, this is impossible for arbitrary DAGs with arbitrary netflow vectors, so we must restrict to a special class of flow polytopes. We define a class of \emph{strongly planar} DAGs with netflow vectors $(\G,\a)$, generalizing an existing definition in the one-source-one-sink case~\cite{MMS,BM}.
We will use such a strongly planar pair $(\G,\a)$ to induce a unimodular \emph{framing triangulation} on the flow polytope $\F_G(\a)$. Moreover, we will generalize the framing lattices of~\cite{vBC} by providing a poset structure on the simplices of the framing triangulation.

The netflow vector $\a$ of a strongly planar pair $(\G,\a)$ may have arbitrary integer values, but we begin our arguments by using decontraction moves to reduce to the case when each source vertex of $\G$ has netflow $1$, each sink vertex has netflow $-1$, and each internal vertex has netflow 0 (preserving strong planarity along the way). We call such DAGs \emph{strongly planar balanced DAGs} with \emph{unit netflow vectors} and write $\F_\G(1)$ for the unit flow polytope.
In this setting, there must be $m$ sources $\{s_1,\dots,s_m\}$ and $m$ sinks $\{t_1,\dots,t_m\}$. The planar embedding of $\G$ induces an ordering of the sources (resp. sinks) from bottom to top.
We define a \emph{layering} of $\G$ to be a collection $\p=\{p_1,\dots,p_m\}$ of routes (paths from source to sink) such that each $p_i$ travels from the $s_i$ to $t_i$, and none cross each other transversally. 
We prove that the map
\[\J:\p\mapsto\sum_{i=1}^m\I(p_i)\]
from a layering of $\G$ to the sum of indicator vectors of its routes defines a bijection between layerings of $\G$ and lattice points of $\F_\G(1)$; we call $\J(\p)$ the \emph{indicator vector} of $\p$.
Moreover, we define a condition of \emph{compatibility} on layerings: two layerings $\p$ and $\q$ are compatible if no two of their routes cross each other transversally, and without loss of generality $p_i$ lies weakly below $q_i$ for all $i\in[m]$.

A \emph{layering-clique} is a set of pairwise compatible layerings. Given a layering-clique $\K$, its \emph{layering-clique simplex} $\D_1(\K)$ is the convex hull of the indicator vectors of layerings of $\K$. The first main result of the paper is that the layering-clique simplices of maximal layering-cliques model a unimodular triangulation of $\F_\G(1)$ which we call the \emph{framing triangulation} of $\F_\G(1)$:

\begin{thmIntro}[{Theorem~\ref{thm:triangulation}, Corollary~\ref{cor:triangulation}}]
	\label{thm:A}
	If $\mathcal C$ is the set of layering-cliques of a strongly planar balanced DAG $\G$, then
	$\{\Delta_1(\K)\ :\ \K\in\mathcal C\}$
	is a unimodular triangulation of $\F_\G(1)$.
\end{thmIntro}

Our next goal is to study the dual graph of the framing triangulation of a strongly planar balanced DAG $\G$ with the intent of obtaining some analog of framing lattices
In the one-source-one-sink case, adjacent simplices of the framing triangulation are related by \emph{rotation} of a route~\cite{WIWT,vBC}. In our new setting, we show that adjacent simplices may be related by one of three types of moves on layering-cliques: rotation, shuffle, or realignment.
Moreover, these moves may be given directions so that if $\K$ and $\L$ are adjacent maximal layering-cliques, then $\L$ is either an up-rotation, up-shuffle, up-realignment, down-rotation, down-shuffle, or down-realignment of $\K$.
In the latter three cases we write $\L\prec\K$. 
\begin{thmIntro}
	The transitive closure of these relations is a poset structure $\preceq$ on the maximal layering-cliques of $\G$, and hence on the simplices of the framing triangulation of $\F_\G(1)$.
\end{thmIntro}
We call this order the \emph{framing poset} of $\G$. This generalizes the framing lattices of planar-framed DAGs in the one-source-one-sink case~\cite{vBC}.
Unlike the one-source-one-sink case, our framing posets may have multiple greatest and/or least elements, and hence are not lattices (though we suspect that adding a greatest element $\hat1$ and least element $\hat0$ would make them into a lattice).

\begin{exIntro}
	Figure~\ref{fig:small_example} shows an example of a strongly planar balanced DAG $\G$ on the top-left (with all edges oriented from left to right). On the bottom-left is its three-dimensional flow polytope $\F_\G(1)$ with its 9 lattice points labelled by the corresponding layering. On the right are the 8 maximal layering-cliques of $\G$ arranged into the Hasse diagram of the framing poset. Each maximal layering-clique consists of 4 layerings (blue, magenta, green, and orange): the corresponding four vertices of $\F_\G(1)$ form a three-dimensional simplex of the framing triangulation of $\F_\G(1)$. For example, the bottom maximal layering-clique along with its corresponding simplex of $\F_\G(1)$ are highlighted. In this way we obtain the framing triangulation of $\F_\G(1)$ into 8 simplices.
\begin{figure}
	\centering
	\def\svgscale{.19}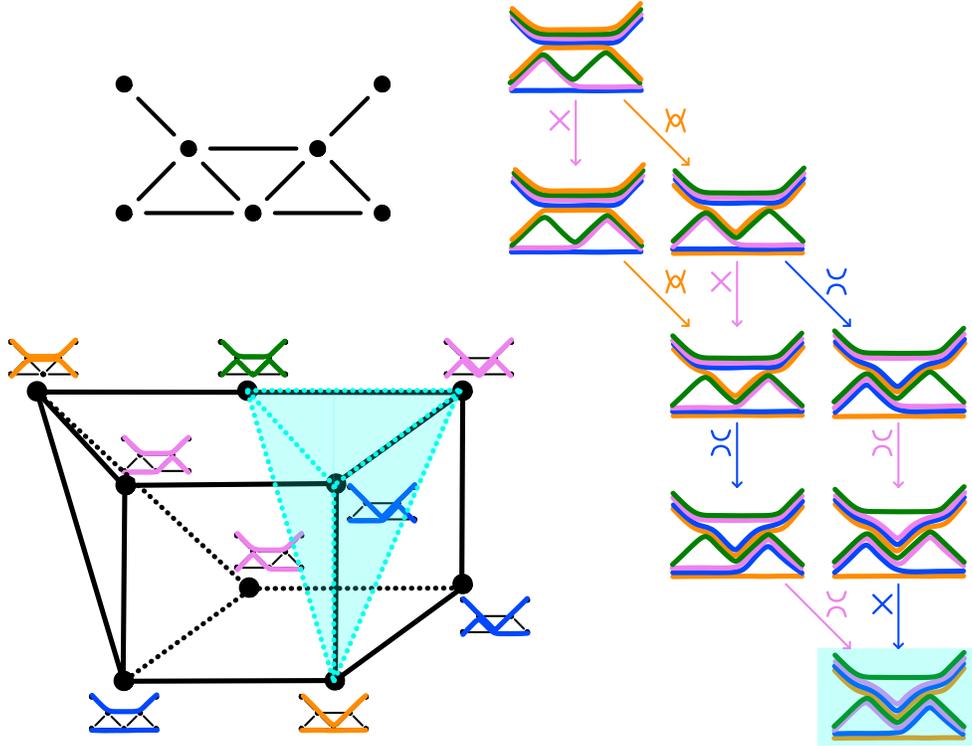
	\caption{A strongly planar balanced DAG with its unit flow polytope and maximal layering-cliques.}
	\label{fig:small_example}
\end{figure}
\end{exIntro}

Finally, we treat the possibility of extending framing triangulation results beyond the strongly planar case.
We give an example of a DAG with integer netflow vector whose flow polytope has no lattice triangulation induced by a pairwise compatibility condition on its integer points, indicating that any theory of framing triangulations for arbitrary flow polytopes must go beyond pairwise compatibility conditions. On the other hand, we expect that framing triangulation results similar to those in this paper and in the one-source-one-sink literature exist for a wider class of flow polytopes.

The structure of the paper is as follows.
In Section~\ref{sec:back1}, we give background on flow polytopes and on framing triangulations in the one-source-one-sink case.
In Section~\ref{sec:strong-balanced} we define the {strongly planar} DAGs with netflows $(\G,\a)$ which we aspire to triangulate. We then use decontraction steps to reduce to the case of a strongly planar balanced DAG with unit netflow vector.
In Section~\ref{sec:framingtriangulations} we define layerings and layering-cliques and prove that these layering-cliques combinatorially describe triangulations on unit flow polytopes of strongly planar balanced DAGs. Much of this work is done by adding a community source and sink in order to apply results from the one-source-one-sink case.
In Section~\ref{sec:mutations} we describe the rotation, shuffle, and realignment moves which connect adjacent maximal layering-cliques, and in Section~\ref{sec:poset} we use these local moves to define the framing poset. Finally, in Section~\ref{sec:obs} we give an example of a DAG with integer netflow vector whose flow polytope admits no lattice triangulation and we conclude the article with some directions for future research.

We conclude the introduction by acknowledging a concurrent project by Gonz\'alez D'Le\'on, Hanusa, and Yip~\cite{DHY} which, among other things, generalizes the theory of framing triangulations to flow polytopes from DAGs with only one negative netflow entry. In particular, their theory involves a notion of \emph{route-matchings} analogous to our layerings, which form a simplicial complex of \emph{cliques} analogous to our layering-clique complex. 
The theories of framing triangulations developed in the present article and~\cite{DHY} agree for strongly planar DAGs with only one sink (i.e., the framed DAGs which are studied in both papers).

\subsection{Acknowledgments}

The author was supported by the NSF grant DMS-2054255. The author thanks Maxwell Hosler for useful discussions.
The author also thanks the anonymous reviewers of Combinatorial Theory, who (among other things) pointed out an issue with the initial version of Section~\ref{sec:obs} which has been corrected in the present document; instead of erroneously claiming that the flow polytope of Example~\ref{ex:k33} has no lattice triangulations, we now claim that it has no lattice triangulations \emph{induced by a pairwise compatibility condition on the vertices}.

\subsection{Notation and Conventions}

If $\alpha$ is an edge of a directed graph, we write $t(\alpha)$ for the beginning (tail) vertex of $\alpha$ and $h(\alpha)$ for the end (head) vertex of $\alpha$. We notate paths left-to-right, so that if $\alpha$ and $\beta$ are edges with $h(\alpha)=t(\beta)$ then $\alpha\beta$ is the path using first $\alpha$, then $\beta$. In almost all cases, when drawing DAGs in figures we will not draw arrow-heads to indicate directions of edges; all edges should be presumed to travel left-to-right in the absence of an arrow-head.
As a convention, the symbol $G$ will refer to a DAG with no choice of embedding, while $\G$ refers to a DAG with a fixed choice of embedding.

\section{Flow Polytopes}
\label{sec:back1}

In this section we give some basic background on triangulations and integral equivalences of lattice polytopes. Then we will define
 DAGs, netflow vectors, and the resulting flow polytopes. We will recall the framing triangulations of unit flow polytopes induced by framings on DAGs with one source and one sink.

\subsection{Some lattice polytope basics}

Following~\cite[\S2]{MMS}, we say that two lattice polytopes $P\subseteq\mathbb R^n$ and $Q\subseteq\mathbb R^m$ are \emph{integrally equivalent} if there is an affine transformation $\phi:\mathbb R^n\to\mathbb R^m$ whose restriction to $P$ is a bijection from $P$ to $Q$ that preserves the lattice, i.e., such that $\phi$ is a bijection between $\mathbb Z^n\cap\text{aff}(P)$ and $\mathbb Z^m\cap\text{aff}(Q)$ where $\text{aff}(-)$ denotes affine span. We also say that the map $\phi$ is an \emph{integral equivalence} and that $P\ucong Q$. Integral equivalence is a notion of ``isomorphism'' on lattice polytopes which is sometimes referred to as \emph{unimodular equivalence} in the literature.

We concern ourselves with triangulations of lattice polytopes arising as flow polytopes.

\begin{defn}\label{defn:triangulation}
	Let $P$ be a $d$-dimensional lattice polytope. A \emph{(lattice) triangulation} of $P$ is a set $\mathcal T$ of simplices such that
	\begin{enumerate}
		\item $P=\cup_{\Delta\in\mathcal T}\Delta$,
		\item if $\Delta\in\mathcal T$ and $\Delta'$ is a face of $\Delta$, then $\Delta'\in\mathcal T$, and
		\item for any $\Delta_1,\Delta_2\in\mathcal T$, the intersection $\Delta_1\cap\Delta_2$ is a (possibly empty) face of both $\Delta_1$ and $\Delta_2$.
	\end{enumerate}
	The triangulation $\mathcal T$ is \emph{unimodular} if each simplex of $\mathcal T$ has normalized volume 1 within its affine span.
\end{defn}

\subsection{DAGs and flow polytopes}

A graph $G=(V,E)$ is a collection of vertices $V$ and a collection of directed edges $E$ between vertices of $V$. The graph $G$ is a \emph{directed acyclic graph (DAG)} if it has no oriented cycles. A \emph{source} of $G$ is a vertex with no incoming edges, and a \emph{sink} of $G$ is a vertex with no outgoing edges. For convenience, we will always assume that every vertex $V$ is incident to at least one edge, so that no vertex is both a source and sink of $G$.
Let $G=(V,E)$ be a directed acyclic graph (DAG) on the vertex set $V=[n]$ and let $\a:=(a_1,\dots,a_n)\in\mathbb Z^{n}$ be a vector of integer weights on the vertex set (with vertex $i$ having weight $a_i$).
We call $\a$ a \emph{netflow vector} of $G$.
A function $F:E\to\mathbb R_{\geq0}$ is a (nonnegative) \emph{$\a$-flow} (or merely a \emph{flow} if the netflow vector $\a$ is understood from the context) if for each $i\in[n]$, the equation
\[
	\sum_{(\alpha:i\to j)\in E}F(\alpha)-\sum_{(\alpha:k\to i)\in E}F(\alpha)=a_i
\]
is satisfied.
The flow $F$ is \emph{integer} (resp. rational) if all coordinates $F(\alpha)$ are integers (resp. rational).

\begin{defn}
	The \emph{$\a$-flow polytope} $\F_G(\a)$ is the polytope consisting of nonnegative $\a$-flows of $G$.
\end{defn}

When $\a$ is an integer vector, all vertices of the $\a$-flow polytope $\F_G(\a)$ are integral points.
It is reasonable to make the following assumptions about our DAG $G$ and netflow vector $\a$:
\begin{enumerate}
	\item If $i$ is a source vertex, then $a_i$ is positive.
	\item If $i$ is a sink vertex, then $a_i$ is negative.
	\item There exists an $\a$-flow of $G$.
\end{enumerate}
If (1) (resp. (2)) is not satisfied, then any edge incident to vertex $i$ cannot be given positive flow, hence vertex $i$ along with all incident edges may be deleted from $G$ without affecting the set of $\a$-flows. If (3) is not satisfied, then the flow polytope $\F_G(\a)$ is empty. If (1), (2), and (3) are all satisfied, we say that $\a$ is a \emph{nondegenerate} netflow vector of $G$. 

\begin{lemma}\label{lem:netflow_sum_zero}
	If $\a$ is a nondegenerate netflow vector of $G$, then $\sum_{i\in[n]}a_i=0$.
\end{lemma}
\begin{proof}
	Choose $F$ to be an $\a$-flow of $G$. Then
	\begin{align*}
		\sum_{i\in[n]}a_i&=\sum_{i\in[n]}\left(\sum_{(\alpha:i\to j)\in E}F(\alpha)-\sum_{(\alpha:k\to i)\in E}F(\alpha)\right).
	\end{align*}
	In the sum on the right, for any edge $\alpha:j\to k$ the term $F(\alpha)$ appears once indexed by vertex $j$ and $-F(\alpha)$ appears once indexed by vertex $k$. These terms then cancel out for each edge and the sum is 0.
\end{proof}

For the rest of the paper, we will always assume that all netflow vectors are nondegenerate.

	Intuitively, flow polytopes model flows through networks. For example, one may imagine each edge of $G$ to be a pipe, and the label $F(\alpha)$ to be an amount of water (or any type of flow) traveling through that pipe in the direction of the edge.
	The equation $\sum_{(\alpha:i\to j)\in E}F(\alpha)-\sum_{(\alpha:k\to i)\in E}F(\alpha)=a_i$ at any vertex $i$ translates to the idea that flow is conserved moving through that vertex, and $a_i$ flow is added to the system at that node (or $|a_i|$ is subtracted if $a_i<0$).

\subsection{Framing triangulations for DAGs with one source and one sink}
\label{ssec:backonesource}

In this section, we recall a unimodular triangulation of $\F_G(1)$ given by Danilov, Karzanov, and Koshevoy~\cite{DKK} when $G$ is a DAG with one source and one sink and $\a=(1,0,\dots,-1)$.

For this subsection, $G$ will always be a DAG with one source and one sink. A \emph{nonnegative flow on $G$ of strength $c$} is a flow with netflow vector $(c,0,\dots,0,-c)$ (where $c$ labels the source, and $-c$ labels the sink).
A \emph{unit flow} is a nonnegative flow on $G$ of strength $1$.
We use the symbol $\hat0$ for the source vertex and $\hat1$ for the sink vertex; all other vertices are \emph{internal vertices}.
The \emph{unit flow polytope} $\F_G(1)$ is the polytope of unit flows. More generally, if $m\in\mathbb Z_{\geq1}$ then $\F_G(m\times 1)$ is the polytope of flows of strength $m$, or the dilation of the unit flow polytope by $m$.
Vertices of $\F_G(1)$ are precisely the indicator vectors of \emph{routes} of $G$, i.e., paths from source to sink.
If $v$ is a vertex of $G$, let $\inn(v)$ be the set of incoming edges to $v$ and let $\out(v)$ be the set of outgoing edges of $v$.

\begin{defn}\label{defn:framed-dag}
	Let $G=(V,E)$ be a DAG with one source and one sink. For each internal vertex $v$ of $G$, assign a linear order to the edges in $\inn(v)$ and assign a linear order to the edges in $\out(v)$. This assignment is called a \emph{framing} of $G$, which we denote by $\R$. We call a DAG $G$ with a framing $\R$ a \emph{framed DAG}. If $e$ is less than $f$ in the linear order for $\R$ on $\inn(v)$, we write $e\prec_{\R,\inn(v)}f$ (and similarly for $\out(v)$). When $\R$ and/or $\inn(v)$ or $\out(v)$ is clear, we may drop one or both subscripts.
\end{defn}

We notate a framing $\R$ by labelling every internal half-edge of the DAG $G$ with a number. See Figure~\ref{fig:square} for an example: we have $\beta_1\prec_{\R,\out(v)}\beta_2$ because the tail-label of $\beta_1$ is lower than that of $\beta_2$.

A framing on a DAG induces a notion of pairwise compatibility on its routes:

\begin{defn}\label{defn:compat0}
	Let $v$ be an internal vertex of a framed DAG $(G,\R)$. We define the \emph{post-$v$-order} $\prec_v^+$ on the set of paths from $v$ to the sink $\hat1$ of $G$.
	Let $p$ and $q$ be distinct paths from $v$ to the sink.
	Let $\sigma=\alpha_1\dots\alpha_m$ be the maximal common subpath of $p$ and $q$ beginning at $v$. Say $p$ contains $\sigma\beta$ and $q$ contains $\sigma\gamma$, where $\beta$ is less than $\gamma$ in $\prec_{\R,\text{out}(h(\alpha_m))}$.
	In this case, we say that $p\prec_v^+q$.
	This defines a total order on paths from $v$ to a sink.

	Dually, we define the \emph{pre-$v$-order} $\prec_v^-$ on the set of paths from a source to $v$.
	If $p$ and $q$ are distinct paths from a source to $v$, then let $\sigma=\alpha_1\dots\alpha_m$ be the maximal common subpath of $p$ and $q$ ending at $v$.
	If $p$ contains $\beta\sigma$ and $q$ contains $\gamma\sigma$, where $\beta\prec_{\R,\inn(t(\alpha_1))}\gamma$, then $p\prec_v^-q$.
\end{defn}

If $p$ and $q$ are routes of $(G,\R)$ which both contain an internal vertex $v$, then we say that $p\prec_v^+q$ if $p_v^+\prec_v^+q_v^+$, where $p_v^+$ (resp. $q_v^+$) is the subpath of $p$ (resp. $q$) from $v$ to the sink. If $p$ and $q$ agree after the vertex $v$, then $p=_v^-q$ (even if $p$ and $q$ differ before the vertex $v$). We may similarly write $p\prec_v^-q$ or $p=_v^-q$.

\begin{defn}\label{defn:compat0.5}
	Let $p$ and $q$ be routes of $(G,\R)$ which both contain a vertex $v$. The routes $p$ and $q$ are \emph{incompatible} at $v$ if, without loss of generality, $p\prec_v^-q$ and $q\prec_v^+p$. The routes $p$ and $q$ are \emph{incompatible} if they are incompatible at any shared vertex, otherwise they are \emph{compatible}.
	A \emph{clique} is a set of pairwise compatible routes of $G$.
\end{defn}

Equivalently, $p$ and $q$ are incompatible if without loss of generality there exists a subpath $\alpha_1 R\beta_2$ of $p$ and $\alpha_2 R\beta_1$ of $q$ such that $\alpha_1\prec_{\inn(s(R))}\alpha_2$ and $\beta_1\prec_{\out(t(R))}\beta_2$.
We say that in this case $\sigma$ is an \emph{incompatibility} of $p$ and $q$.
Note that if $p$ and $q$ agree to the left of a vertex $v$, then $p=_v^-q$ and they cannot be incompatible at $v$.

The following definition and theorem connect cliques on $G$ to the unit flow polytope $\F_G(1)$.

\begin{defn}
	If $\K$ is a clique of $(G,\R)$, then a \emph{($\K$-)clique combination} of $G$ is a linear combination
	\[\sum_{p\in\K}a_p\I(p),\]
	where each $a_p\geq0$. It is \emph{positive} if each $a_p>0$, and \emph{unit} if $\sum_{p\in\K}a_p=1$.
	The set of unit flows arising as (necessarily unit) $\K$-clique combinations is the \emph{clique simplex} $\Delta_1(\K)$.
\end{defn}

\begin{thm}[{\cite[Theorem 1]{DKK}}]\label{thm:DKK}
	Let $F$ be a nonnegative flow of a framed DAG $(G,\R)$ with netflow vector $(S,0,\dots,0,-S)$. Then there is a unique positive clique combination $F=\sum_{p\in\K}a_p\I(p)$ for $F$. Moreover, if $F$ is integer-valued, then all coefficients $a_p$ are integers.
\end{thm}

\begin{remk}
	Note that the phrasing of~\cite[Theorem 1]{DKK} promises (through different terminology) uniqueness of clique combinations, while our phrasing in Theorem~\ref{thm:DKK} promises uniqueness of {positive} clique combinations. For example, in the framed DAG of Figure~\ref{fig:square} the point $\I(\alpha_2\beta_2)$ may be realized through the clique combinations $\I(\alpha_2\beta_2)$ or $\I(\alpha_2\beta_2)+0\I(\alpha_1\beta_1)$. The uniqueness promised in~\cite[Theorem 1]{DKK} considers these to be the same decomposition. We aim to sidestep such quibbles by phrasing the result as uniqueness of \emph{positive} clique combinations.
\end{remk}

Theorem~\ref{thm:DKK} implies that the clique simplices of maximal cliques of $(G,\R)$ form a unimodular triangulation of $\F_G(1)$ (in particular, unimodularity follows from the final statement of the theorem). We call this the \emph{framing triangulation} of $\F_G(1)$ (it is also referred to as the \emph{DKK triangulation} in the literature).
See Example~\ref{ex:square}.
Note that it is also shown in~\cite{DKK} that these framing triangulations are regular, though we do not treat this in the present paper.

\begin{example}\label{ex:square}
	See the framed DAG of Figure~\ref{fig:square}. Arrows are labelled in black and the framing at the unique internal vertex $v$ is labelled in red. Its flow polytope is integrally equivalent to a square, shown on the right of the figure.
	The routes $\alpha_1\beta_2$ and $\alpha_2\beta_1$ are incompatible at the vertex $v$, as $\beta_1\prec_{v}^+\beta_2$ but $\alpha_1\prec_{v}^-\alpha_2$. Any other choice of two routes is compatible. It follows that the two maximal cliques of $G$ are $\{\alpha_1\beta_1,\alpha_1\beta_2,\alpha_2\beta_2\}$ and $\{\alpha_1\beta_1,\alpha_2\beta_1,\alpha_2\beta_2\}$, shown on the middle of the figure. Each maximal clique $\K$ corresponds to a triangle (simplex) of the flow polytope whose vertices are the indicator vectors of the routes of $\K$; see the framing triangulation of the flow polytope on the right of the figure.
	\begin{figure}
		\centering
		\def\svgscale{.38}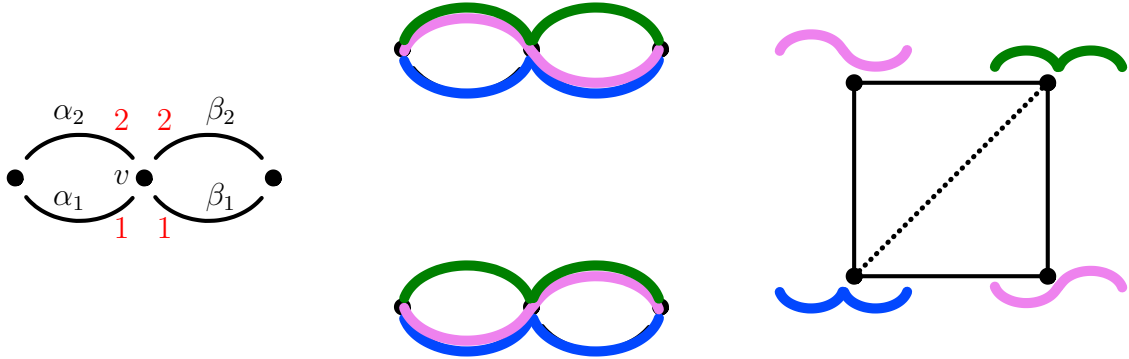
		\caption{A framed DAG, its two maximal cliques, and its framing-triangulated flow polytope.}
		\label{fig:square}
	\end{figure}
\end{example}

The simplices of the framing triangulation of the unit flow polytope of a framed DAG $(G,\R)$ with one source and one sink were given the structure of a \emph{framing lattice} in~\cite{vBC}. 
Our framing posets will agree with these framing lattices in the case of planar-framed one-source-one-sink DAGs.
On the other hand, we do not directly need these results of~\cite{vBC} to obtain our framing posets, so we will not give this construction as background.

\section{Strongly Planar and Balanced DAGs}
\label{sec:strong-balanced}

We start this section by defining strongly planar DAGs with netflows $(\G,\a)$. These give the class of flow polytopes which we will triangulate in this paper.
We will then use some reduction steps to show that it suffices to triangulate \emph{unit flow polytopes} of strongly planar \emph{balanced} DAGs.

\subsection{Strongly planar DAGs}

Mimicking~\cite[Definition 5.3]{BM} (which applies only to the one-source-one-sink case), we define strong planarity of a DAG.

\begin{defn}\label{defn:splanar}
	A planar embedding $\G$ of a nondegenerate DAG $G$ is \emph{strongly planar} if
	\begin{enumerate}
		\item for every directed edge $(i,j)$ of $\G$, the x-coordinate of $i$ is strictly less than the x-coordinate of $j$,
		\item each edge $(i,j)$ is embedded into the plane as the graph of a piecewise differentiable function of $x$,
		\item every source and every sink of $\G$ is incident to the exterior face of the planar embedding of $\G$, and
		\item every source vertex of $\G$ has the same x-coordinate, and every sink vertex of $\G$ has the same x-coordinate.
	\end{enumerate}
	We say that a nondegenerate DAG $G$ is strongly planar if it admits a strongly planar embedding.
\end{defn}

Condition (4) of Definition~\ref{defn:splanar} does not affect the class of DAGs which admit strongly planar embeddings, but is included for our own convenience.
In the following, we will use the symbol $\G$ to refer to a DAG $G$ equipped with a strongly planar embedding.
When $\G$ is equipped with a netflow vector, we make one extra requirement.

\begin{defn}
	Let $\G$ be a strongly planar DAG with nondegenerate netflow vector $\a$. We say that $(\G,\a)$ is \emph{strongly planar} if every vertex $i\in V$ with a nonzero netflow $a_i\neq0$ is incident to the exterior face of the planar embedding of $\G$.
\end{defn}

Figure~\ref{fig:not_strongly_planar} shows three embeddings of DAGs $\G$ with netflow vector $\a$ (labelled in blue) such that $(\G,\a)$ fail to be strongly planar. On the left, there is an edge $v_1\to v_2$ but $v_1$ is to the right of $v_2$, violating condition (1) for strong planarity of $\G$ (note that one could change the embedding to obtain a strongly planar pair). On the middle, the edge from $v_1$ to $v_2$ is not a function of $x$ as it is not 1-1, violating condition (2) for strong planarity of $\G$. The DAG $\G$ on the right is strongly planar, but $(\G,\a)$ is not strongly planar because the vertex $v_2$ has nonzero netflow $a_{v_2}=1$ while being incident only to interior faces.

\begin{figure}
	\centering
	\def\svgscale{.38}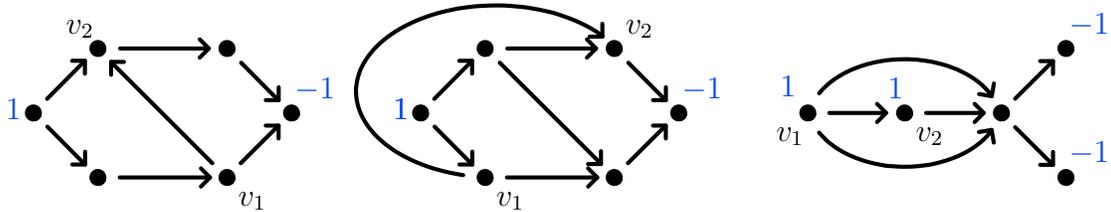
	\caption{Three examples of pairs $(\G,\a)$ which fail to be strongly planar. Netflows are labelled in blue; vertices with no integer label have netflow 0.}
	\label{fig:not_strongly_planar}
\end{figure}

Note that strong planarity uses the information of $\a$ as well as $\G$. For example, the right of Figure~\ref{fig:not_strongly_planar} is not strongly planar but one may edit the netflow vector $\a$ so that $\a_{v_1}=2$ and $\a_{v_2}=0$ while using the same embedding to make $(\G,\a)$ strongly planar and nondegenerate.

\subsection{Planar framed DAGs with one source and one sink}
\label{ssec:planar-framing}

A special case of Theorem~\ref{thm:DKK} which is relevant to this paper is when $G$ is strongly planar in addition to having one source and one sink.
In this case, a choice of strongly planar embedding $\G$ of $G$ induces a \emph{planar framing} $\R_\P$ of $\G$ defined such that $\alpha\prec_{\R_\P,\inn(v)}\beta$ if and only if $\alpha$ is drawn below $\beta$ in the strongly planar embedding, and $\gamma\prec_{\R_\P,\out(v)}\delta$ if and only if $\gamma$ is drawn below $\delta$ in the strongly planar embedding. For example, the framing labelled in Figure~\ref{fig:square} is the planar framing.

When $\G$ is a strongly planar DAG with one source and one sink, we consider it to be equipped with the planar framing $\R_\P$ by default. For example, two routes of $\G$ are \emph{compatible} if they are compatible as routes of $(G,\R_\P)$, and a \emph{clique} of $\G$ is a clique of the planar framing $(G,\R_\P)$.

\begin{remk}\label{remk:compat_planar}
	Routes of a planar-framed DAG $\G$ are compatible if and only if they may be drawn on the graph without crossing each other transversally. For example, in the DAG of Figure~\ref{fig:square}, the only two routes which cross are $\alpha_1\beta_2$ and $\alpha_2\beta_1$, meaning that these are the only two routes which are incompatible with each other. Cliques of $\G$ are then collections of routes with the condition that no two cross each other.
\end{remk}

\subsection{Balanced DAGs}

We now define balanced DAGs.

\begin{defn}\label{defn:balanced}
	Let $G$ be a DAG. Recall that we assume that every vertex of $G$ is incident to at least one edge, hence no vertex is both a source and a sink. Let $\a_1$ be the \emph{unit netflow vector} such that 
	\[(\a_1)_i=\begin{cases}1&i\text{ is a source}\\-1&i\text{ is a sink}\\0&\text{else.}\end{cases}\]
	We say that $G$ is \emph{balanced} if $\a_1$ is nondegenerate.
	When $G$ is a balanced DAG, a \emph{unit flow} on $G$ is an $\a_1$-flow. For any nonnegative constant $c$, a \emph{nonnegative flow on $G$ of strength $c$} is a $(c\a_1)$-flow.
	If $G$ is a balanced DAG, then we say its \emph{unit flow polytope} $\F_G(1):=\F_G(\a_1)$ is the polytope of unit flows of $G$.
\end{defn}

\begin{lemma}\label{lem:sources-sinks-num}
	If $G$ is a balanced DAG, then $G$ has the same number of sources as sinks.
\end{lemma}
\begin{proof}
	Follows from Lemma~\ref{lem:netflow_sum_zero}.
\end{proof}
If $G$ has one source and one sink, then its unit flow polytope in the above definition agrees with its unit flow polytope as defined in 
Section~\ref{ssec:backonesource}.

\begin{remk}
	This paper deals heavily with DAGs which are strongly planar and balanced. Note that when $G$ is a balanced DAG with nondegenerate netflow vector $\a$ and $\G$ is a planar embedding of $G$, then $\G$ is strongly planar if and only if $(\G,\a)$ is strongly planar because condition (3) of strong planarity of DAGs ensures the condition of strong planarity of the pair $(\G,\a)$. As such, in the following we will use the phrase ``strongly planar balanced DAG'' without fear of ambiguity.
\end{remk}

\subsection{Reducing to the strongly planar balanced case}

The goal of this paper is to give unimodular triangulations for the flow polytopes of strongly planar pairs $(\G,\a)$.
We now show that it will suffice to give unimodular triangulations for flow polytopes of strongly planar balanced DAGs.
We do this by taking a strongly planar pair $(\G,\a)$ and decontracting edges until we reach a strongly planar balanced DAG whose unit flow polytope is integrally equivalent to $\F_\G(\a)$.

\begin{defn}\label{defn:decont}
	Let $\G$ be a planar-embedded DAG with nondegenerate netflow vector $\a$ such that $(\G,\a)$ is strongly planar. We now define a DAG $G'=(V',E')$ with a planar embedding $\G'$ by making the following additions to $(V,E)$:
	\begin{enumerate}
		\item The vertex set $V'$ consists of the vertex set $V$ of $\G$ along with, for every vertex $i\in V$ with nonzero netflow $a_i$, the \emph{decontracted vertices} $\{v_1^i,\dots,v_{|a_i|}^i\}$. In other words,
	\[V'=V\cup\left(\bigcup_{i\in V\ :\ a_i\neq0}\{v_1^i,\dots,v_{|a_i|}^i\}\right).\]
\item The edge set $E'$ consists of the edge set $E$ of $\G$ along with, for every decontracted vertex $v_j^i\in V'$, the \emph{decontracted edge} $v_j^i\to i$ (if $a_i>0$) or $i\to v_j^i$ (if $a_i<0$).
	\end{enumerate}
	Since any vertex $i\in V$ with $a_i\neq0$ is on the exterior face of $\G$, we may choose a strongly planar embedding $\G'$ of $G'$ (there are many such choices).
	See Figure~\ref{fig:reduction_to_balanced}, where a strongly planar $(\G,\a)$ is shown on top (with $\alpha$ labelled in blue) and one possible choice of $\G'$ on bottom. We equip $\G'$ with the unit netflow vector $\a'_1$.
\end{defn}

\begin{figure}
	\centering
	\def\svgscale{.38}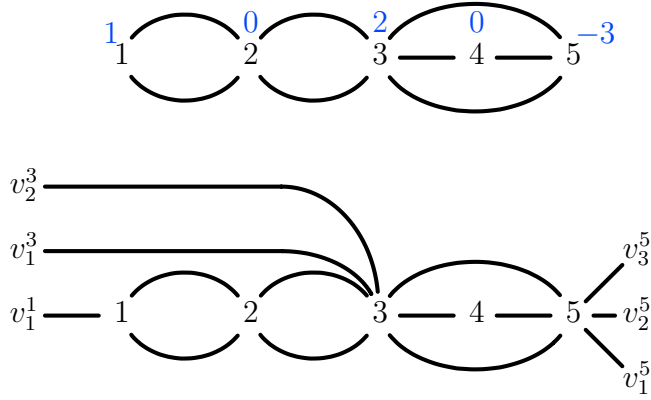
	\caption{Reducing to the strongly planar balanced case. Vertices are drawn as black letters (or symbols); on the top, netflow is labelled in blue.}
	\label{fig:reduction_to_balanced}
\end{figure}

\begin{prop}\label{prop:reduce_to_balanced}
	If $(\G,\a)$ is a strongly planar DAG, then there exists a strongly planar balanced DAG $\G'$ such that $\F_\G(\a)\ucong\F_{\G'}(1)$ is an integral equivalence.
\end{prop}
\begin{proof}
	We show that, with the notation of Definition~\ref{defn:decont}, $\G'$ is a strongly planar balanced DAG and $\F_{\G'}(1)\ucong\F_\G(\a)$.
	Recall the unit netflow vector $\a'_1$ on $\G'$ by
	\[(\a')_i=\begin{cases}1&i\text{ is a source}\\-1&i\text{ is a sink}\\0&\text{else.}\end{cases}\]
		It is immediate that every unit flow of $(\G',\a'_1)$ labels every decontracted edge with 1. 
		Moreover, it is immediate that a vector of $\mathbb R^{E'}$ is a unit flow of $(\G',\a'_1)$ if and only if its projection to $\mathbb R^E$ is an $\a$-flow.
		It follows that projecting from $\mathbb R^{E'}$ to $\mathbb R^E$ by forgetting all coordinates of decontraction edges gives an integral equivalence from $\F_{\G'}(1)$ to $\F_\G(\a)$.
\end{proof}

Geometrically, $\F_{\G'}(1)$ is a copy of $\F_{\G}(\a)$ placed in the higher-dimensional space $\mathbb R^{E'}$ at height 1 in all new coordinates of the decontracted edges.

\begin{remk}
	We have some choice in where we embed the decontracted vertices and edges in $\G'$.
	For example, Figure~\ref{fig:reduction_to_balanced_alternate} gives an alternate embedding of $\G'$ as a strongly planar balanced DAG.
	These different choices will lead to different framing triangulations in the coming sections.
\begin{figure}
	\centering
	\def\svgscale{.38}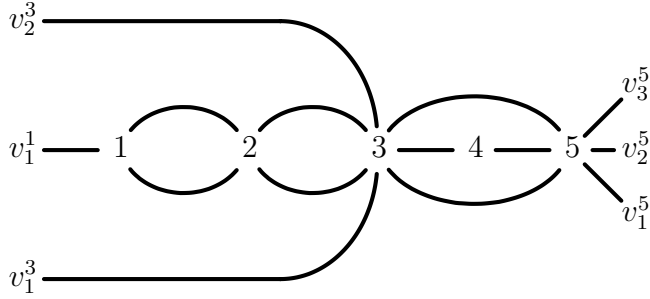
	\caption{A different way to reduce to the strongly planar balanced case.}
	\label{fig:reduction_to_balanced_alternate}
\end{figure}
\end{remk}

In light of Proposition~\ref{prop:reduce_to_balanced}, from now on we will restrict ourselves to the case where $\G$ is a strongly planar balanced DAG to obtain unimodular triangulations. Given any strongly planar $(\G,\a)$, one may use Proposition~\ref{prop:reduce_to_balanced} to obtain a strongly planar balanced DAG $\G'$ with $\F_{\G'}(1)\ucong\F_\G(\a)$ and pull the unimodular triangulation on $\F_{G'}(1)$ through the equivalence to triangulate $\F_G(\a)$.
See Example~\ref{ex:nonunit} for an example of this process.

We finish this section with a useful observation about strongly planar balanced DAGs. If $\G$ is strongly planar and balanced, then Lemma~\ref{lem:sources-sinks-num} shows that $\G$ has $m$ sources and $m$ sinks, for some positive integer $m$.
Since all source vertices have the same x-coordinate, we may notate them $s_1,\dots,s_m$ from bottom to top. Similarly, we notate the sink vertices $t_1,\dots,t_m$ bottom to top. See, for example, Figure~\ref{fig:layering}.
We keep this notation for the rest of this article.

\section{Framing Triangulations of Strongly Planar Balanced DAGs}
\label{sec:framingtriangulations}

In this section, we construct the framing triangulation of $\F_\G(1)$ given a strongly planar balanced DAG $\G$.
We will model the integer points of $\F_\G(1)$ by \emph{layerings} on $\G$ and define a notion of pairwise compatibility on layerings which induces the desired triangulation.
Our proofs will use the \emph{two-point extension} $\hat\G$ obtained from $\G$ by adding a community source and a community sink vertex to $\G$. The two-point extension will allow us to apply the triangulation results of Danilov, Karzanov, and Koshevoy~\cite{DKK} in the unit one-source-one-sink case to $\hat\G$ in order to obtain a unimodular triangulation on $\F_\G(1)$.

For this section, fix a strongly planar balanced DAG $\G$ with $m$ sources and $m$ sinks.
As indicated in the previous section, the symbol $\G$ includes not only the data of the DAG $G$ but also a choice of strongly planar embedding.
Recall from Definition~\ref{defn:balanced} that a \emph{unit flow} on $\G$ is a flow with the unit netflow vector $\a_1$, and that a \emph{nonnegative flow} on $\G$ is any nonnegative scaling of a unit flow.

\subsection{Two-point extensions}

We define the two-point extension of a balanced DAG, allowing us to relate the unit flow polytope of a balanced DAG with the unit flow polytope of a one-source-one-sink DAG.

\begin{defn}
	Let $\G$ be a strongly planar balanced DAG. Define the \emph{two-point extension} $\hat\G$ as follows:
	\begin{enumerate}
		\item The vertex set of $\hat\G$ is $\hat V:=V\cup\{\hat0\}\cup\{\hat1\}$.
		\item The edge set of $\hat\G$ is
			$\hat E:=E\cup\{(\hat0,s_i)\ :\ i\in[m]\}\cup\{(t_i,\hat1)\ :\ i\in[m]\}$.
	\end{enumerate}
	We call the edges of $\hat E\backslash E$ the \emph{extension edges} of $\hat\G$. These are precisely the edges of $\hat E$ which are incident to the source $\hat0$ or sink $\hat1$.
	Intuitively we obtain $\hat\G$ from $\G$ by adding a new source $\hat0$ and connecting it to all sources of $G$, and connecting all sinks of $G$ to a new sink $\hat1$. By construction, $\hat0$ is the unique source of $\hat\G$ and $\hat1$ is the unique sink of $\hat\G$.
	Embed the vertex $\hat0$ to the left of every vertex of $\G$ and embed $\hat1$ to the right of every vertex of $\G$, so that $\hat\G$ is strongly planar.
	See Figure~\ref{fig:two_point_extension}: a strongly planar balanced DAG $\G$ is shown on the left, and its two-point extension $\hat\G$ on the right.
\end{defn}

\begin{figure}
	\centering
	\def\svgscale{.35}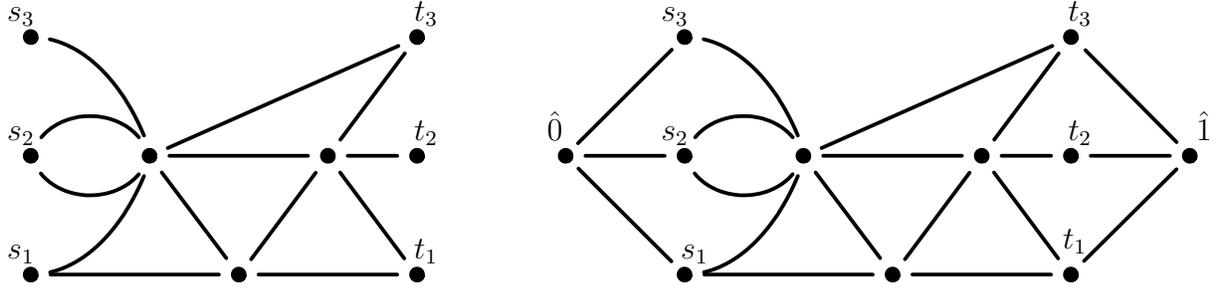
	\caption{The two-point extension of a strongly planar balanced DAG.}
	\label{fig:two_point_extension}
\end{figure}

We now connect certain flows and paths on $\G$ with certain flows and paths on $\hat\G$.

\begin{defn}
	A \emph{(nonnegative) layered flow} on $\hat\G$ is a nonnegative flow $F$ on $\hat\G$ such that if $\alpha$ and $\beta$ are extension edges of $\hat\G$, then $F(\alpha)=F(\beta)$.
\end{defn}

\begin{defn}
	Let $F$ be a nonnegative flow on $\G$ of strength $S$.
	Define a nonnegative flow $\hat F$ on $\hat\G$ by
	\[\hat F(\alpha)=\begin{cases}
		F(\alpha) & \alpha\in E\\
		S & \alpha\in\hat E\backslash E.
	\end{cases}\]
\end{defn}

\begin{lemma}\label{lem:flows_g_to_hat}
	The map $F\mapsto\hat F$ is a bijection from nonnegative flows on $\G$ to nonnegative layered flows on $\hat\G$. If $F$ has strength $S$, then $\hat F$ has strength $(m\times S)$.
\end{lemma}
\begin{proof}
	It is immediate that if $F$ is a nonnegative flow on $\G$ then $\hat F$ is a nonnegative flow on $\hat\G$, and that the map $F\mapsto\hat F$ is injective.
	Moreover, given any layered flow $R'$ on $\hat\G$, the restriction $R:=R'|_E$ is a nonnegative flow on $\G$ satisfying $\hat R=R'$, so the map $F\mapsto\hat F$ is a bijection. The final statement follows because there are $m$ source vertices and $m$ sink vertices of $\G$.
\end{proof}

\begin{remk}\label{remk:poly1}
	Setting $S=1$ in Lemma~\ref{lem:flows_g_to_hat} shows that the map $F\mapsto\hat F$ bijects the flows of $\F_\G(1)$ with the layered flows of $\F_{\hat\G}(m\times 1)$ (i.e., to the layered flows of the dilation of the unit flow polytope $\F_{\hat\G}(1)$ by $m$). Said another way, the unit flow polytope $\F_\G(1)$ is (up to integral equivalence) equal to the intersection of $\F_{\hat\G}(m\times1)$ with the hyperplanes setting flow through each extension edge of $\hat\G$ to 1.
	Note that this is not a face of $\F_{\hat\G}(m\times1)$ when $m>1$.
\end{remk}

We now wish to compare routes of $\hat\G$ to routes of $\G$.

\begin{defn}
	A \emph{route} $p$ of $\G$ is a directed path from a source of $\G$ to a sink of $\G$.
	The \emph{indicator vector} of $p$ is the function $\I(p):E\to\mathbb R$ defined by \[\I(p)(\alpha)=\begin{cases}1&\alpha\in p\\0&\alpha\not\in p.\end{cases}\]
\end{defn}

Let $p$ be a route of $\G$. Coming out of the sink vertex $t(p)$ is a unique extension edge $e_t$ of $\hat E$; similarly, there is a unique extension edge $e_s$ of $\hat E$ ending at $s(p)$. Define the route $\hat p:=e_spe_t$ of $\hat\G$. It is immediate that the map $p\mapsto\hat p$ is a bijection from routes of $\G$ to routes of $\hat\G$ (note that routes of $\hat\G$ correspond to vertices of $\F_{\hat\G}(1)$, but routes of $\G$ do not correspond to vertices of $\F_\G(1)$ unless $\G$ already has one source and one sink). See Figure~\ref{fig:two_point_extension_routes}.

\begin{figure}
	\centering
	\def\svgscale{.35}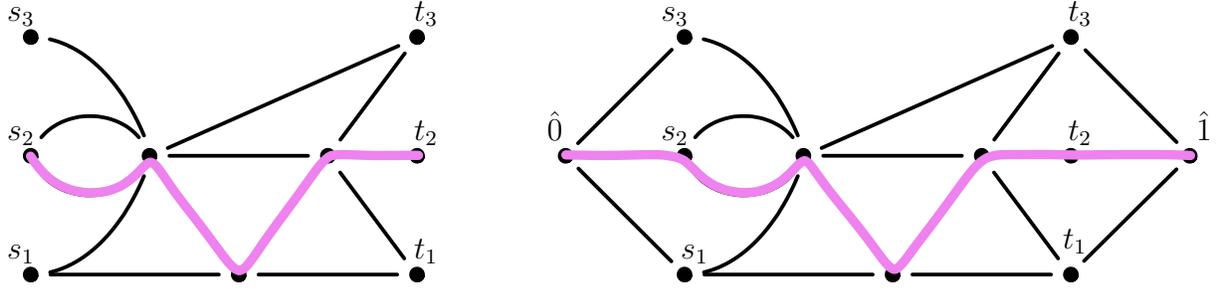
	\caption{On the left is a route $p$ of $\G$ and on the right is the corresponding route $\hat p$ of $\hat\G$.}
	\label{fig:two_point_extension_routes}
\end{figure}

We remark now that since $\hat\G$ is strongly planar and has one source vertex and one sink vertex, we may equip it with the planar framing $\R_\P$ as in Section~\ref{ssec:planar-framing}. This gives us a notion of pairwise compatibility of routes of $\hat\G$, and we may calculate the clique complex of $\hat\G$ to triangulate $\R_{\hat\G}(1)$.
We now define compatibility of routes of a strongly planar balanced DAG in a way which agrees with compatibility of the corresponding routes of its two-point extension (equipped with the planar framing).
Compare with Definition~\ref{defn:compat0}.

\begin{defn}\label{defn:compat1}
	Let $v$ be a vertex of a strongly planar balanced DAG $\G$. We define the \emph{post-$v$-order} $\prec_v^+$ on the set of paths from $v$ to a sink of $\G$.
	Let $p$ and $q$ be distinct paths from $v$ to a sink.
	Let $\sigma=\alpha_1\dots\alpha_m$ be the maximal common subpath of $p$ and $q$ beginning with $v$. Say $p$ contains $\sigma\beta$ and $q$ contains $\sigma\gamma$, where $\beta$ is below $\gamma$ in the planar embedding of $\G$.
	In this case, we say that $p\prec_v^+q$.
	This defines a total order on paths from $v$ to a sink.

	Dually, we define the \emph{pre-$v$-order} $\prec_v^-$ on the set of paths from a source to $v$.
	If $p$ and $q$ are distinct paths from a source to $v$, then let $\sigma=\alpha_1\dots\alpha_m$ be the maximal common subpath of $p$ and $q$.
	If $p$ contains $\beta\sigma$ and $q$ contains $\gamma\sigma$, where $\beta$ is below $\gamma$ in the planar embedding of $\G$, then $p\prec_v^-q$.
\end{defn}

If $p$ and $q$ are routes of $\G$ which both contain the internal vertex $v$, then we say that $p\prec_v^+q$ if $p_v^+\prec_v^+q_v^+$, where $p_v^+$ (resp. $q_v^+$) is the subpath of $p$ (resp. $q$) from $v$ to a sink. If $p$ and $q$ agree after the vertex $v$, then $p=_v^+q$ (even if $p$ and $q$ differ before the vertex $v$). We may similarly write $p\prec_v^-q$ or $p=_v^-q$.

\begin{defn}
	Let $p$ and $q$ be routes of $\G$ and let $v$ be a vertex contained in both $p$ and $q$. The routes $p$ and $q$ are \emph{incompatible} at $v$ if, without loss of generality, $p\prec_v^-q$ and $q\prec_v^+p$. The routes $p$ and $q$ are \emph{incompatible} if they are incompatible at any shared vertex, otherwise they are \emph{compatible}.
	A \emph{clique} is a set of pairwise compatible routes of $\G$.
\end{defn}

\begin{remk}
	The notions of $\prec_v^+$ and compatibility agree with Definitions~\ref{defn:compat0} and~\ref{defn:compat0.5} applied to the two-point extension $\hat\G$. In particular, if $p$ and $q$ are routes of a strongly planar balanced DAG $\G$ and $v$ is a shared vertex, then we have
	\begin{enumerate}
		\item $p\prec_v^+q$ if and only if $\hat p\prec_v^+\hat q$,
		\item $p\prec_v^-q$ if and only if $\hat p\prec_v^-\hat q$, and hence
		\item $p$ is compatible with $q$ if and only if $\hat p$ is compatible with $\hat q$.
	\end{enumerate}
\end{remk}

Two routes of a strongly planar balanced DAG are compatible if and only if they do not cross each other transversally.
If $p$ and $q$ are routes of $\G$, then we say that $p$ is \emph{weakly above} $q$ (in the planar embedding of $\G$) if for any x-coordinate at which $p$ and $q$ are defined, the y-coordinate of $p$ is greater than or equal to the y-coordinate of $q$. We also say that $q$ is \emph{weakly below} $p$ in this case. If $p$ and $q$ are any routes of $\G$ then the condition that $p$ and $q$ are compatible amounts to the condition that $p$ is either weakly above or weakly below $q$. Hence, if $\K$ is a clique of $\G$, the routes of $\K$ are totally ordered from bottom to top by the planar embedding of $\G$.

\subsection{Layerings}

We now define layerings on the strongly planar balanced DAG $\G$ and show that they are in bijection with integer unit flows on $\G$. We define the notion of compatibility between layerings which will induce a unimodular triangulation on $\F_\G(1)$.

\begin{defn}\label{defn:horizontal}
	A route of $\G$ is \emph{horizontal} if it starts at the source $s_i$ and ends at the sink $t_i$ for some $i\in[m]$.
	We say that it is a \emph{horizontal $i$-route}.
	A route $\hat p$ of $\hat\G$ is horizontal if its corresponding route $p$ of $\G$ is horizontal.
\end{defn}

\begin{defn}\label{defn:layering}
	A \emph{layering} of $\G$ is a set $\p=\{p_1,\dots,p_m\}$ of pairwise compatible routes of $\G$ such that $p_i$ is a horizontal $i$-route for all $i\in[m]$.
\end{defn}

The right of Figure~\ref{fig:layering} shows a layering of $\G$.

\begin{figure}
	\centering
	\def\svgscale{.38}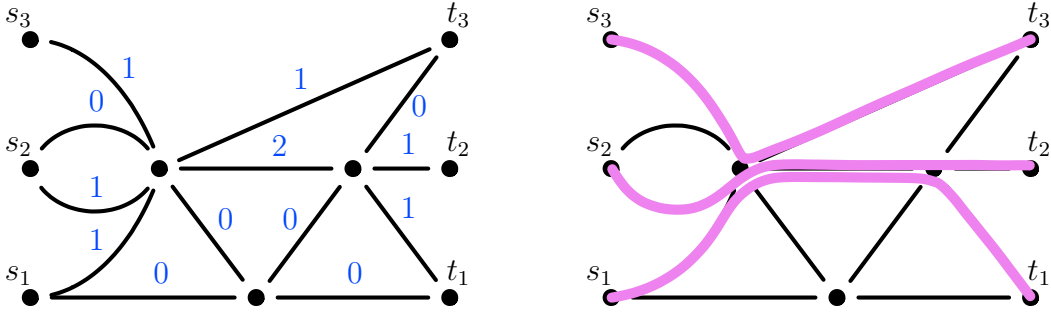
	\caption{A strongly planar balanced DAG with a layering in magenta (right) and the corresponding unit flow in blue (left).}
	\label{fig:layering}
\end{figure}

\begin{remk}
	The condition that the routes of a layering are pairwise compatible amounts to the condition that the routes $\{\hat p_1,\dots,\hat p_m\}$ form a clique of $\hat\G$. This is the same as the condition that none of the routes $\{p_1,\dots,p_m\}$ cross each other transversally.
\end{remk}

We now prepare to connect layerings to integer unit flows.

\begin{prop}\label{lem:layered-horizontal}
	Let $\hat F$ be a layered flow of $\hat\G$ with rational coordinates. Then the clique combination realizing $\hat F$ uses only horizontal routes.
\end{prop}
\begin{proof}
	We first prove the statement assuming that the layered flow $\hat F$ has integer coordinates. If $\hat F$ is the zero vector then its clique combination uses no routes and the result is trivial, so we assume that $\hat F$ is not the zero vector. Let $S\in\mathbb Z_{>0}$ such that $\hat F(\alpha)=S$ for any source or sink edge $\alpha$, so that $\hat F$ is a strength-$(mS)$ flow of $\hat\G$.

	By Theorem~\ref{thm:DKK}, the flow $\hat F$ gives rise to a unique clique $\hat{\K}$ of $\hat\G$ and a decomposition $\hat F=\sum_{\hat p\in\hat{\K}}a_{\hat p}\I({\hat p})$, where each coefficient $a_{\hat p}$ is a positive integer. This is the clique combination realizing $\hat F$, and we wish to show that each route $\hat p$ of $\hat\K$ is horizontal. Since $\hat F$ has strength $mS$, we have $\sum_{{\hat p}\in\hat{\K}}a_{\hat p}=mS$. Then we may get rid of the coefficients by rewriting the clique combination as $\hat F=\sum_{i=1}^{mS}\I(\hat p_i)$, where each ${\hat p}\in\hat{\K}$ appears in the multiset $\{{\hat p}_1,\dots,{\hat p}_{mS}\}$ with multiplicity $a_{\hat p}$.

	Since $\hat F(\alpha)=S$ for each source edge $\alpha$ of $\hat\G$, for any $i\in[m]$ exactly $S$ of the routes $\{{\hat p}_j\}$ pass through the vertex $s_i$. Similarly, $S$ of the routes $\{\hat p_j\}$ pass through $t_i$. Order the routes $\{\hat p_1,\dots,\hat p_{mS}\}$ so that the routes $\{\hat p_{(i-1)S+1},\dots,\hat p_{(i-1)S+S}\}$ pass through vertex $s_i$, with $\hat p_{(i-1)S+1}\preceq_{s_i}^+\hat p_{(i-1)S+2}\preceq_{s_i}^+\dots\preceq_{s_i}^+\hat p_{(i-1)S+S}$. 
	We claim that for all $i\in[m]$, the routes $\{\hat p_{(i-1)S+1},\dots,\hat p_{(i-1)S+S}\}$ pass through vertex $t_i$, so that all routes of $\hat{\K}$ are horizontal. Suppose to the contrary and take $i\in[m]$ minimal such that one of the routes $\{\hat p_{(i-1)S+1},\dots,\hat p_{(i-1)S+S}\}$ does not pass through $t_i$ -- say $\hat P:=\hat p_{(i-1)S+l}$ for some $l\in[S]$. By minimality of $i$, for any $i'<i$ the $S$ routes $\{\hat p_{(i'-1)S+1},\dots,\hat p_{(i'-1)S+S}\}$ pass through $t_{i'}$. Since $\hat F(\alpha_{i'})=S$ for $\alpha_{i'}$ the sink edge of $\G$ incident to $t_{i'}$, this means that $\hat P$ cannot also pass through $t_{i'}$. This shows that $\hat P$ must pass through $t_{i'}$ for some $i'>i$. Symmetrically, there must exist some $i''\in[m]$ and $j\in[S]$ with $i''>i$ such that $\hat Q:=\hat p_{(i''-1)S+j}$ ends at $t_i$. Then $\hat P$ begins at $i$ and ends at $i'$, while $\hat Q$ begins at $i''>i$ and ends at $i<i'$. Then $\hat P$ and $\hat Q$ must cross each other transversally in the strongly planar embedding of $\hat\G$, hence they are incompatible, contradicting the assumption that they are part of the same clique $\hat{\K}$. This completes the proof that all routes of $\hat{\K}$ are horizontal when $\hat F$ has integer coordinates.

	Now, suppose that $\hat F$ is a layered flow with rational coordinates. Let $c$ be the GCD of all entries of $F$. Then $c\cdot\hat F$ is a layered flow with integer coordinates, so by our proof of the integer-valued case it is realized as a clique combination 
	$c\cdot\hat F=\sum_{\hat p\in\hat K}a_{\hat p}\I(\hat p)$ where $\hat K$ is a clique of horizontal routes. Then
	\[
		\hat F=\frac{1}{c}(c\cdot\hat F)=\sum_{\hat p\in\hat K}\left(\frac{a_{\hat p}}{c}\right)\I(\hat p)\]
	is the clique combination for $\hat F$ and uses only horizontal routes, completing the proof.
\end{proof}

\begin{lemma}\label{prop:layerings_and_flows}
	The map $\J:\p\mapsto\sum_{i\in[m]}\I(p_i)$ is a bijection from layerings of $\G$ to integer unit flows on $\G$.
\end{lemma}
We call $\J(\p)$ the \emph{indicator vector} of the layering $\p$. Note that this is a different symbol than $\I$, used for the indicator vector of a route.
Figure~\ref{fig:layering} shows a layering $\p$ on the right and labels the corresponding unit flow $\J(\p)$ on the left.
\begin{proof}
	It is immediate that if $\p$ is a layering of $\G$ then $\J(\p)$ is an integer unit flow on $\G$, so $\J$ is well-defined as a map from layerings of $\G$ to integer unit flows on $\G$.

	We now argue injectivity of $\J$. Suppose that $\p$ and $\q$ are two layerings such that $\J(\p)=\J(\q)$. Since $\p$ is a layering, the routes $\{p_1,\dots,p_m\}$ of $\G$ are compatible, hence the routes $\{\hat p_1,\dots,\hat p_m\}$ of $\hat\G$ are compatible; then $\sum_{i=1}^m\I(\hat p_i)$ is a clique combination for ${\J(\p)}$. Similarly, $\sum_{i=1}^m\I(\hat q_i)$ is a clique combination for ${\J(\q)}={\J(\p)}$. Then the uniqueness part of Theorem~\ref{thm:DKK} implies that $\{\hat p_1,\dots,\hat p_m\}=\{\hat q_1,\dots,\hat q_m\}$ and hence $\p=\q$. We have shown that $\J$ is injective.

	We now show that $\J$ is surjective. Let $F$ be a unit flow on $\G$ with integer coordinates.
	By Lemma~\ref{lem:flows_g_to_hat}, the function $\hat F$ is a layered flow on $\hat\G$ of strength $m$, so that if $\alpha$ is a source or sink edge of $\hat\G$ then $\hat F(\alpha)=1$.
	Then by Theorem~\ref{thm:DKK}, one obtains $\hat F=\sum_{i=1}^{m'}a_i\I(\hat p_i)$ for some clique $\{\hat p_1,\dots,\hat p_{m'}\}$ of $\hat\G$ and some coefficients $a_i>0$.
	Since $F$ is integer-valued, the flow $\hat F$ is integer-valued, hence all coefficients $a_i$ are integers.
	By Proposition~\ref{lem:layered-horizontal}, each route $\hat p_i$ is horizontal. Since $\hat F$ is a layered flow of strength $m$, we have $\hat F(\alpha)=1$ for all source or sink edges $\alpha$, showing that all coefficients $a_i$ must equal 1 and $m'=m$.
	It follows that $\p:=\{p_1,\dots,p_m\}$ is a layering (where $p_i$ is the route of $\G$ corresponding to the route $\hat p_i$ of $\hat\G$) and it is now immediate that $\J(\p)=F$, proving surjectivity of $\J$.
\end{proof}

\subsection{The layering-clique complex}

We now define the compatibility relation on layerings which will induce our framing triangulations.

\begin{defn}\label{defn:route-compatible}
If $\K$ is a subset of the layerings of $\G$ then define
$\routes(\K):=\cup_{\p\in\K}\left(\cup_{p\in\p}p\right)$ to be the set of routes used in layerings of $\K$.
	We say that $\K$ is \emph{route-compatible} if $\routes(\K)$ is a clique of $\G$. Similarly, two layerings $\p$ and $\q$ are \emph{route-compatible} if $\{\p,\q\}$ is route-compatible.
\end{defn}

\begin{defn}\label{defn:layering-compatible}
	Let $\p$ and $\q$ be layerings of $\G$.
	We say that $\p\prec\q$ if $\p$ and $\q$ are route-compatible and $p_i\preceq_{s_i}^+ q_i$ for all $i\in[m]$.
	We say that $\p$ and $\q$ are \emph{compatible} if $\p\preceq\q$ or $\q\preceq\p$; otherwise, they are \emph{incompatible}.
	A \emph{layering-clique} ${\K}$ is a set of pairwise compatible layerings.
	The \emph{layering-clique complex} is the simplicial complex of layering-cliques.
\end{defn}

Note that the layering-clique complex is indeed an abstract simplicial complex because a subset of a layering-clique is a layering-clique.

Note that under the condition that $\p$ and $\q$ is route-compatible, no two routes appearing in $\routes(\{\p,\q\})$ may cross each other. Then the condition that $p_i\prec_{s_i}^+q_i$ is the same as the condition that $p_i$ lies weakly below $q_i$ in the planar embedding of $\G$, which is the same as the condition that $p_i\prec_{t_i}^-q_i$. So, the apparent asymmetry of using the orders $\prec_{s_i}^+$ instead of $\prec_{t_i}^-$ is unimpactful because of the assumption of route-compatibility.

The relation $\prec$ is not transitive on general layerings of $\G$. See Figure~\ref{fig:transitive}; if $(\p,\q,\r)$ are the blue, magenta, and green layerings, respectively, then $\p\prec\q$ and $\q\prec\r$ but $\p\not\prec\r$.

\begin{figure}
	\centering
	\def\svgscale{.38}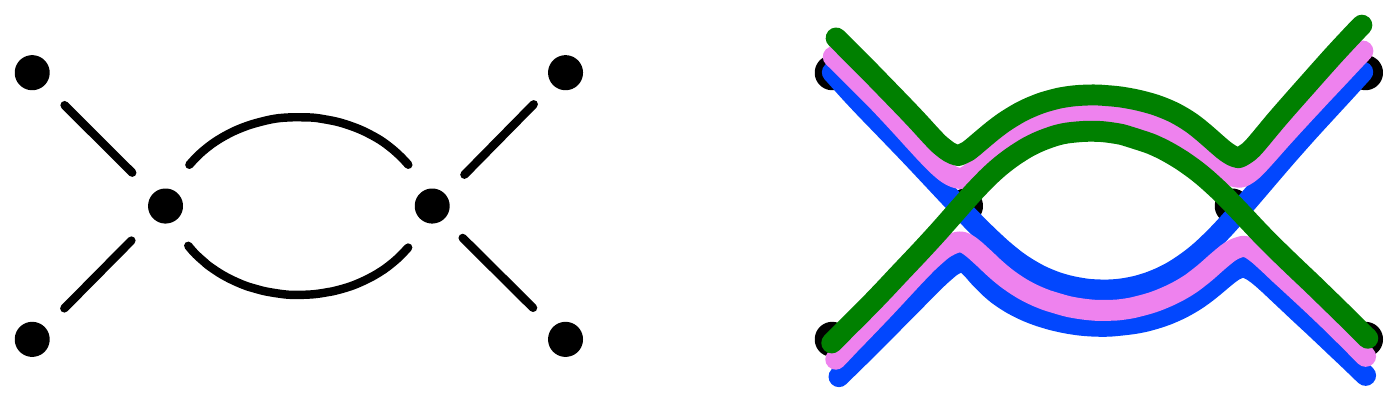
	\caption{The relation $\prec$ is not transitive on layerings.}
	\label{fig:transitive}
\end{figure}

On the other hand, if $\K=\{\p^1,\dots,\p^{\l}\}$ is any route-compatible set of layerings then it is immediate that $\prec$ restricts to a partial order on $\K$.
These layerings then form a layering-clique if and only if $\prec$ is a total order on this set.
In this case, we write $\prec_\K$ for the restriction of $\prec$ to the layering-clique $\K$.
We commonly notate elements of a layering-clique as ${\K}=\{\p^1,\p^2,\dots,\p^l\}$, where we always assume that $\p^1\prec\p^2\prec\dots\prec\p^l$.

We now connect layering-cliques of $\G$ with simplices of $\F_\G(1)$.

\begin{defn}
	If $\K$ is a layering-clique of a strongly planar balanced DAG $\G$, then a \emph{($\K$-)layering-clique combination} of $\G$ is a linear combination
	\[\sum_{\p\in\K}a_\p\J(\p),\]
	where each $a_p\geq0$. It is \emph{positive} if each $a_p>0$, and \emph{unit} if $\sum_{\p\in\K}a_\p=1$.
	Note that the layering-clique combination is unit if and only if the resulting sum is a unit flow.
	The set of unit flows arising as (necessarily unit) $\K$-clique combinations is the \emph{layering-clique simplex} $\Delta_1(\K)$.
\end{defn}

We finally phrase an analog of Theorem~\ref{thm:DKK} in the setting of strongly planar balanced DAGs.

\begin{thm}\label{thm:triangulation}
	Let $F$ be a nonnegative flow of a strongly planar balanced DAG $\G$. Then there exists a unique layering-clique combination
	$F=\sum_{\p\in{\K}}a_\p\J(\p)$ for $F$. Moreover, if $F$ is integer-valued, then all coefficients $a_\p$ are integers.
\end{thm}

Note that we do not prove that the framing triangulations are regular, though this is shown in the unit one-source-one-sink case in~\cite{DKK}. We do not have an example of a strongly planar balanced DAG whose framing triangulation is not regular.

\begin{proof}
	Since $\F_\G(1)$ is a rational polytope, it suffices to prove the result only for rational-valued flows $F$.
	Moreover, it suffices to prove the desired result only for integer-valued flows, and the result for rational-valued flows follows after multiplying by the GCD.

	Let $F$ be an integer-valued flow of strength $S$.
	By Lemma~\ref{lem:flows_g_to_hat}, the function $\hat F$ is a strength-$(mS)$ layered flow of $\hat\G$.
	By Theorem~\ref{thm:DKK}, the flow $\hat F$ gives rise to a unique clique $\hat{\K}$ of $\hat\G$ and a decomposition $\hat F=\sum_{\hat p\in\hat{\K}}a_{\hat p}\I({\hat p})$, where each coefficient $a_{\hat p}$ is a positive integer. Since $\hat F$ has strength $mS$, we have $\sum_{{\hat p}\in\hat{\K}}a_{\hat p}=mS$. Then we may rewrite the decomposition of $\hat F$ as $\hat F=\sum_{i=1}^{mS}\I(\hat p_i)$, where each ${\hat p}\in\hat{\K}$ appears in $\{{\hat p}_1,\dots,{\hat p}_{mS}\}$ with multiplicity $a_{\hat p}$.

	Since $\hat F(\alpha)=S$ for each source edge $\alpha$ of $\hat\G$, for any $i\in[m]$ exactly $S$ of the routes $\{{\hat p}_j\}$ pass through the vertex $s_i$.
	Order the routes $\{\hat p_1,\dots,\hat p_{mS}\}$ such that the routes $\{\hat p_{(i-1)S+1},\dots,\hat p_{(i-1)S+S}\}$ pass through vertex $s_i$, with $\hat p_{(i-1)S+1}\preceq_{s_i}^+\hat p_{(i-1)S+2}\preceq_{s_i}^+\dots\preceq_{s_i}^+\hat p_{(i-1)S+S}$.
	By Proposition~\ref{lem:layered-horizontal}, each route $\hat p_j$ is horizontal.

	For $j\in[S]$, define $\p^j:=\{p_j,p_{S+j},p_{2S+j},\dots,p_{(m-1)S+j}\}$ (so that $p^j_k=p_{(k-1)S+j}$).
	Since $\hat\K=\{\hat p_1,\dots,\hat p_{mS}\}$ is a clique and $\hat p_{(i-1)S+j}$ is a horizontal route of $\hat\G$ passing through $s_i$, each $\p^j$ is a layering of $\G$.
	Moreover, we now show that any two layerings $\p^i$ and $\p^j$ are compatible. Suppose without loss of generality that $i<j$.
	The routes $\{\hat p^i_1,\dots,\hat p^i_m,\hat p^j_1,\dots,\hat p^j_m\}$ are a subset of the clique $\hat{\K}$ of $\G$, hence themselves form a clique. This shows that $\p^i$ and $\p^j$ are route-compatible. Moreover, by choice of ordering of $p_i$'s, for any $k\in[m]$ we have $p^i_k=p_{(k-1)S+i}\prec_{s_i}^+p_{(k-1)S+j}=p^j_k$, so $\p^i$ and $\p^j$ are compatible with $\p^i\prec\p^j$.
	This shows that all layerings of $\{\p^1,\dots,\p^S\}$ are pairwise compatible.
	Finally, restricting the equality $\hat F=\sum_{i=1}^{mS}\I(\hat p_i)$ to the edges of $\G$ gives $F=\sum_{i=1}^{mS}\I(p_i)$, hence
	\begin{align*}
		\sum_{i=1}^S\J(\p^i)
					&=\sum_{i=1}^S\big(\sum_{j=1}^m\I(p^i_j)\big)\\
					&=\sum_{i=1}^S\big(\sum_{j=1}^m\I(p_{(j-1)S+i})\big)\\
					&=\sum_{i=1}^{mS}\I(p_i)\\
					&=F,
	\end{align*}
	realizing the flow $F$ as a sum of integer vectors of layerings. Finally, let ${\K}$ be the underlying set of the multiset $\{\p^1,\dots,\p^S\}$ and for $\p\in{\K}$, let $a_\p$ be the number of times $\p$ appears as some $\p^i$. This realizes our layering-clique combination
	\[F=\sum_{i=1}^{S}\J(\p^i)=\sum_{\p\in{\K}}a_\p\J(\p).\]

	It remains to show that no two distinct positive layering-clique combinations may realize the same integer flow. To this end, suppose for contradiction that we have an integer flow $F$ of $\G$ and that there exist distinct layering-clique combinations
	\begin{equation}\label{eqgg}F=\sum_{i=1}^{l_A}a_{\p^i}\J(\p^i)=\sum_{i=1}^{l_B}b_{\q^i}\J(\q^i),\end{equation}
	where $A=\{\p^1,\dots,\p^{l_A}\}$ and $B=\{\q^1,\dots,\q^{l_B}\}$ are layering-cliques. Index so that $\p^1\prec\dots\prec\p^{l_A}$ and $\q^1\prec\dots\prec\q^{l_B}$.

	By rearranging the sum in the middle of Equation~\eqref{eqgg}, we achieve
	\[
		F=\sum_{r\ :\ r\in\p^i\text{ for some }i\in[l_A]}\left(\sum_{i\in[l_A]\ :\ r\in\p^i}a_{\p^i}\right)\I(r).
	\]
	Letting $\hat A:=\{\hat r\ :\ r\in\p^i\text{ for some }i\in[l_A]\}$ and putting the above equation through the map $F\mapsto\hat F$ gives
	\begin{equation}\label{eqA}
		\hat F=\sum_{\hat r\in\hat A}\left(\sum_{i\in[l_A]\ :\ r\in\p^i}a_{\p^i}\right)\I(\hat r).
	\end{equation}
	Since $\p$ is a layering, the set $\hat A$ is a clique and hence the above is a positive clique combination for $\hat F$.
	Working with the second clique combination and defining the clique $\hat B$ and coefficients $b_{\q^i}$ analogously, we obtain that
	\begin{equation}\label{eqB}
		\hat F=\sum_{\hat r\in\hat B}\left(\sum_{i\in[l_B]\ :\ r\in\q^i}b_{\q^i}\right)\I(\hat r)
	\end{equation}
	is a positive clique combination for $\hat F$.
By uniqueness of Theorem~\ref{thm:DKK}, the clique combinations~\eqref{eqA} and~\eqref{eqB} must be equal, hence $\hat A=\hat B$ and for any route $\hat r\in\hat A=\hat B$, we have 
	\begin{equation}\label{eqContradiction}\sum_{i\in[l_A]\ :\ r\in\p^i}a_{\p^i}=\sum_{i\in[l_B]\ :\ r\in\q^i}b_{\q^i}.\end{equation}

	Since the layering-clique combinations of~\eqref{eqgg} are not equal, we may take $j\in\min\{l_A,l_B\}$ to be minimal so that $(a_{\p^j},\p^j)\neq(b_{\q^j},\q^j)$. We split into two cases to obtain a contradiction.
	\begin{enumerate}
		\item First, suppose that $\p^j\neq\q^j$.
			Choose $i\in[m]$ such that, without loss of generality, $p^j_i\prec_{s_j}^+q^j_i$.
			Since $\q^j\prec\q^{j'}$ for $j'\geq j$, the route $P:=p^j_i$ does not appear in $\q^{j'}$ for any $j'\geq j$.
			On the other hand, for $j'< j$ we have $\q^{j'}=\p^{j'}$ and $b_{\q^{j'}}=a_{\p^{j'}}$.
			Then
			\begin{align*}
				\sum_{c\in[l_B]\ :\ P\in\q^c}b_{\q^c}
				&=\sum_{c\in[j-1]\ :\ P\in\q^c}b_{\q^c}\\
				&=\sum_{c\in[j-1]\ :\ P\in\p^c}a_{\p^c}\\
				&<\sum_{c\in[j]\ :\ P\in\p^c}a_{\p^c}\\
				&\leq\sum_{c\in[l_A]\ :\ P\in\p^c}a_{\p^c},
			\end{align*}
			contradicting Equation~\ref{eqContradiction}.
		\item Now, suppose that $\p^j=\q^j$ but $a_{\p^j}\neq b_{\q^j}$. Suppose without loss of generality that $a_{\p^j}<b_{\q^j}$. If $j=l_A$, then choose any $i\in[m]$. Otherwise, choose any $i\in[m]$ such that $p^j_i\neq p^{j+1}_i$. Since $\p^{j}\prec\p^{j'}$ for $j'>j$, the route $P:=p^j_i$ is not in $\p^{j'}$ for any $j'>j$. 
			Then we calculate
			\begin{align*}
				\sum_{c\in[l_A]\ :\ P\in\p^c}a_{\p^c}
				&=\sum_{c\in[j]\ :\ P\in\p^c}a_{\p^c}\\
				&=a_{\p^j}+\sum_{c\in[j-1]\ :\ P\in\p^c}a_{\p^c}\\
				&<b_{\q^j}+\sum_{c\in[j-1]\ :\ P\in\p^c}a_{\p^c}\\
				&=b_{\q^j}+\sum_{c\in[j-1]\ :\ P\in\q^c}b_{\q^c}\\
				&\leq\sum_{c\in[l_B]\ :\ P\in\q^c}b_{\q^c},
			\end{align*}
			contradicting Equation~\ref{eqContradiction}.
	\end{enumerate}
	This completes the proof that no two distinct layering-clique combinations may realize the same integer flow of $\F$, and the proof of the theorem.
\end{proof}

Theorem~\ref{thm:triangulation} states that the layering-cliques index a unimodular triangulation of $\F_\G(1)$ (where unimodularity follows from the final statement of the theorem):

\begin{cor}\label{cor:triangulation}
	If $\mathcal C$ is the set of layering-cliques of a strongly planar balanced DAG $\G$, then
	\[
		\{\Delta_1(\K)\ :\ \K\in\mathcal C\}
	\]
	is a unimodular triangulation of $\F_\G(1)$ which we call the \emph{framing triangulation} of $\F_\G(1)$.
\end{cor}
\begin{proof}
	Recall the definition of a triangulation from Definition~\ref{defn:triangulation}. 
	Property (1) is satisfied because every flow $F$ is realized as a clique combination
	\[F=\sum_{\p\in K}a_\p\J(\p),\]
	and hence lies in $\Delta_1(\K)$.
	Property (2) is satisfied because the faces of $\Delta_1(\K)$ are precisely those simplices $\Delta_1(\K')$ where $\K'\subseteq\K$.
	Property (3) is satisfied because $\Delta_1(\K_1)\cap\Delta_1(\K_2)=\Delta_1(\K_1\cap\K_2)$ for layering-cliques $\K_1$ and $\K_2$.

	Finally, for any layering-clique $\K$ the final statement of Theorem~\ref{thm:triangulation} shows that the set of vectors $\{\J(\p)\ :\ \p\in\K\}$ form a $\mathbb Z$-basis of their linear span. Hence, the cone over $\D_1(\K)$ is a unimodular cone within its linear span and $\D_1(\K)$ is itself a unimodular simplex within its affine span.
\end{proof}

\begin{example}
	Figure~\ref{fig:small_example} shows a strongly planar balanced DAG $\G$. It has 8 maximal layering-cliques, shown on the right (arranged into a poset structure which we will explain later). On the bottom-left is the unit flow polytope whose lattice points are labelled by the corresponding layering. Note that the layering labelled in green corresponds to the unique lattice point of $\F_\G(1)$ which is not a vertex. Each maximal layering-clique $\K$ gives rise to a full-dimensional simplex of $\F_\G(1)$ whose vertices are given by the layerings of $\K$. For example, the bottom maximal layering-clique is highlighted along with its corresponding simplex of $\F_\G(1)$. As indicated by Corollary~\ref{cor:triangulation}, these simplices form a unimodular triangulation of $\F_\G(1)$.
	In particular, this implies that $\F_\G(1)$ has normalized volume 8.
\end{example}

Finally, the fact that layering-cliques index a triangulation on $\F_\G(1)$ implies that the layering-clique complex is pure.
\begin{cor}\label{cor:triangulation-cardinality}
	The layering-clique complex of $\G$ is pure, and any maximal layering-clique contains exactly $\dim\F_\G(1)+1$ layerings.
\end{cor}
\begin{proof}
	If $\K$ is a maximal layering-clique of $\G$, then the indicator vectors of layerings of $\K$ form the vertices of the simplex $\Delta_1(\K)$. The simplex $\Delta_1(\K)$ is a full-dimensional simplex of $\F_\G(1)$ by Corollary~\ref{cor:triangulation}, hence has $\dim\F_\G(1)+1$ vertices. This shows that $|\K|=\dim\F_\G(1)$.
\end{proof}

\section{Traveling Between Adjacent Maximal Layering-cliques}
\label{sec:mutations}

We say that two maximal layering-cliques $\K$ and $\L$ of $\G$ are \emph{adjacent} if $|\K\backslash\L|=|\L\backslash\K|=1$. In other words, if the corresponding simplices of the framing triangulation of $\F_{\G}(1)$ share all but one of their vertices.
In this section, we will describe how adjacent maximal layering-cliques must relate to each other: we say that they must relate by a rotation, shuffle, or realignment move.

In the following, let $\G$ be a strongly planar balanced DAG with $m$ sources and $m$ sinks.

\subsection{The dimension zero case} 

We first treat the degenerate case when $\dim\F_\G(1)=0$.

Observe that if $\dim\F_\G(1)=0$, then $\F_\G(1)$ is merely a point. The only triangulation of $\F_\G(1)$ is $\mathcal T=\{\F_\G(1)\}$, consisting of one zero-dimensional simplex. Then by Corollary~\ref{cor:triangulation}, there is precisely one maximal layering-clique of $\G$, and this maximal layering-clique contains only one layering. This proves the following lemma:

\begin{lemma}\label{lem:dim0}
	If $\dim\F_\G(1)=0$ then there is only one maximal layering-clique $\K$ of $\G$, and $|\K|=1$.
\end{lemma}

This section focuses on traveling between adjacent maximal layering-cliques, and Lemma~\ref{lem:dim0} shows that these cannot exist when $\dim\F_\G(1)=0$.
Hence, for the rest of this section we will focus on the case when $\dim\F_\G(1)\geq1$. In this case, Corollary~\ref{cor:triangulation-cardinality} gives us the following convenient result:

\begin{remk}\label{remk:geq1}
	If $\dim\F_\G(1)\geq1$, then any maximal layering-clique of $\G$ contains at least two layerings.
\end{remk}

\subsection{Technical Lemmas}

Before we define rotations, shuffles, and realignments, we need some technical lemmas about the structure of maximal layering-cliques.

\begin{lemma}\label{lem:b}
	Let $\K$ be a maximal layering-clique of $\G$ and let $\p,\q\in\K$ be adjacent layerings in $\prec_\K$ (i.e., such that there is no $\r\in\K$ with $\p\prec\r\prec\q$ or $\q\prec\r\prec\p$). Then there is a (necessarily unique) index $i\in[m]$ such that $p_i\neq q_i$ and $p_j=q_j$ for all $j\neq i$.
\end{lemma}
Note that the existence of adjacent layerings implies that $\dim\F_\G(1)\geq1$ by Lemma~\ref{lem:dim0}.
\begin{proof}
	Suppose without loss of generality that $\p\prec\q$ (so that this is a covering relation in the total order $\prec_\K$ on layerings of $\K$).
	Since $\p$ and $\q$ are distinct, there is at least one value $i'\in[m]$ such that $p_{i'}\neq q_{i'}$. We will show that ${i'}$ must be unique.
	Define $\p':=\{p_1,\dots,p_{{i'}-1},q_{i'},p_{{i'}+1},\dots,p_m\}$ (i.e., $\p'$ is obtained by taking $\p$ and replacing the route $p_{i'}$ with $q_{i'}$).
	Since $\p$ and $\q$ are compatible as layerings, the route $q_{i'}$ is compatible with the route $p_j$ for all $j\in[m]$, so $\p'$ is a valid layering of $\G$.
	Moreover, we now show that the layering $\p'$ is compatible with every layering of $\K$.
	Let $\r$ be any layering of $\K$.
	Since $\p'$ is built from routes $p_i$ and $q_i$, and $\r$ is compatible with both $\p$ and $\q$, it must be true that $\{p'_1,\dots,p'_m,r_1,\dots,r_m\}$ is a clique. Hence, $\r$ and $\p'$ are route-compatible layerings. It remains to show that either $p'_i\preceq_{s_i}^+r_i$ for all $i\in[m]$, or $p'_i\succeq_{s_i}^+r_i$ for all $i\in[m]$.

	Since $\p\prec_\K\q$ is a covering relation of $\prec_\K$, either $\r\preceq_\K\p$ and $\r\preceq_\K\q$ or $\p\preceq_\K\r$ and $\q\preceq_\K\r$. Suppose first that $\r\preceq_\K\p$ and $\r\preceq_\K\q$.
	Then for all $j\in[m]$, we have both $r_j\preceq_{s_j}^+p_j$ and $r_j\preceq_{s_j}^+q_j$. This means that $r_j\preceq_{s_j}^+p'_j$ for all $j\in[m]$, completing the proof that $\p'$ is compatible with $\r$ in this case. The proof in the case when
	$\p\preceq_\K\r$ and $\q\preceq_\K\r$ is the same, but with the directions of all symbols $\preceq$ switched. 

	We have shown that $\p'$ is compatible with every layering of $\K$, hence $\p'\in\K$ by maximality of $\K$. Since $p_{i'}\prec_{s_{i'}}^+q_{i'}=p'_{i'}$, we must have $\p\prec_\K\p'$. Since $\p\prec_\K\q$ is a covering relation, this means that $\p\prec_\K\q\preceq_\K\p'$.
	In particular, for any $j\neq i'$, we have
	\[p_j\preceq_{s_j}^+q_j\preceq_{s_j}^+p'_j=p_j,\]
	hence equality holds throughout. This completes the proof that $p_j=q_j$ for all $j\neq i'$.
\end{proof}

\begin{defn}\label{defn:updown}
	Suppose $\dim\F_\G(1)\geq1$.
	Let $\K$ be a maximal layering-clique and let $\p\in\K$. If $\p$ is not the maximum element of $\prec_\K$, then take $\q$ to be the layering covering $\p$ in $\prec_\K$ and define $\up_\K(\p)$ to be the (unique by Lemma~\ref{lem:b}) index $i\in[m]$ such that $p_i\neq q_i$. If $\p$ is the maximum element of $\prec_\K$, then $\up_\K(\p)$ is undefined. Similarly, if $\p$ is not the minimum element of $\prec_\K$ then take the layering $\r$ covered by $\p$ in $\prec_\K$ and define $\down_\K(\p)$ to be the index $j\in[m]$ such that $p_j\neq r_j$.
\end{defn}

\begin{lemma}
	\label{lem:updown}
	Suppose $\dim\F_\G(1)\geq1$.
	Let $\K=\{\p^1,\dots,\p^S\}$ be a maximal layering-clique and take an index $j\in[S]$.
	The following are equivalent:
	\begin{enumerate}
		\item there exists a route $p^j_i$ of $\p^j$ which appears in no other layering of $\K$,
		\item the set $\{\up_\K(\p^j),\down_\K(\p^j)\}$ consists of a single element which we define as $\updown_\K(\p^j)$, and
		\item the route $p^j_{\updown_\K(\p^j)}$ is the unique route of $\p^j$ which appears in no other layering of $\K$.
	\end{enumerate}
\end{lemma}
Condition (2) amounts to the condition that either $j=0$ (so that $\down_\K(\p^j)$ is not defined), or $j=S$ (so that $\up_\K(\p^j)$ is not defined), or $j\in\{2,\dots,S-1\}$ and $\up_\K(\p^j)=\down_\K(\p^j)$. We then define $\updown_\K(\p^j)$ to be $\down_\K(\p^j)$ or $\up_\K(\p^j)$. If $\up_\K(\p^j)$ and $\down_\K(\p^j)$ both exist but are not equal, then $\updown_\K(\p^j)$ is undefined.
\begin{proof}
	The condition that $\dim\F_\G(1)\geq1$ implies that $S>1$ by Remark~\ref{remk:geq1}.
	We show that (2) is equivalent to both other conditions. First, suppose (2) holds. We consider separately the cases where $j=1$, where $j=S$, and where $j\in\{2,\dots,S-1\}$ and $\up_\K(\p^j)=\down_\K(\p^j)$.
	\begin{enumerate}
		\item Suppose $j=1$. Since $S>1$, the index $\up_\K(\p^1)$ is defined. Then $\updown_\K(\p^1):=\up_\K(\p^1)$; refer to this index as $u$ in the following calculation for brevity. For any $k\in[m]$ such that $k\geq2$, we have $\p^k\succeq\p^2$ and hence $p^k_{u}\succeq_{s_{u}}^+p^2_{u}\succ_{s_u}^+ p^1_{u}$, so $p^1_u$ appears in $\p^k$ precisely when $k=1$. This shows (1). Moreover, by Lemma~\ref{lem:b} we must have $p^1_k=p^1_u$ for all $k\neq u$, hence $u=\updown_\K(\p^1)$ is the unique route of $\p^1$ which appears in no other layering of $\K$, proving (3).
		\item If $j=S$, then a symmetric argument to the above case shows (1) and (3).
		\item Suppose $j\in\{2,\dots,S-1\}$ and $\up_\K(\p^j)=\down_\K(\p^j)$. By Lemma~\ref{lem:b}, for any $k\neq\updown_\K(\p^j)$ we have $p^j_k=p^{j+1}_k=p^{j-1}_k$, hence $p^j_{\updown_\K(\p^j)}$ is the unique route of $\p^j$ which appears in no other layering of $\K$, showing (1) and (3).
	\end{enumerate}
	We have shown that (2) implies (1) and (3). Now suppose (2) fails. Then $j\in\{2,\dots,S-1\}$ and $\up_\K(\p^j)\neq\down_\K(\p^j)$.
	By Definition~\ref{defn:updown}, for $k\in[m]$ we have
	\[
		p^j_k=
		\begin{cases}
			p^{j+1}_k & k\neq\up_\K(\p^j)\\
			p^{j-1}_k & k=\up_\K(\p^j).
		\end{cases}
	\]
	This shows that $p^j_k$ appears in $\p^{j+1}$ or $\p^{j-1}$ for every $k\in[m]$, showing that (1) and (3) both fail. This completes the proof that (2) is equivalent to both (1) and (3).
\end{proof}

\subsection{Shuffles, rotations, and realignments}

We now describe how adjacent maximal layering-cliques $\K$ and $\L$ must relate to each other. We will argue that $\L$ is either a \emph{shuffle}, a \emph{rotation}, or a \emph{realignment} of $\K$.
First, we describe shuffles.

\begin{prop}
	\label{prop:shuffle}
	Suppose $\dim\F_\G(1)\geq1$.
	Let $\K=\{\p^1,\dots,\p^S\}$ be a maximal layering-clique. Suppose for some $i\in[m]$ that every route of $\p^i$ is in another layering of $\K$. Then $i\in\{2,\dots,S-1\}$ and $\up_\K(\p^i)\neq\down_\K(\p^i)$ and the layering
	\[\r=(\p^i\backslash\{p^i_{\up_\K(\p^i)},p^i_{\down_\K(\p^i)}\})\cup\{p^{i+1}_{\up_\K(\p^i)},p^{i-1}_{\down_\K(\p^i)}\}\]
	is compatible with all routes of $\K\backslash\{\p^i\}$ but incompatible with $\p^i$, so that 
	\[\L:=(\K\backslash\{\p^i\})\cup\{\r\}=\{\p^1\prec\dots\p^{i-1}\prec\r\prec\p^{i+1}\prec\dots\prec\p^S\}\]
	is a maximal layering-clique.
\end{prop}
Intuitively, $\r$ is obtained from $\p^i$ by replacing the route $p^i_{\up_\K(\p^i)}$ with $p^{i+1}_{\up_\K(\p^i)}$, and replacing the route $p^i_{\down_\K(\p^i)}$ with the route $p^{i-1}_{\down_\K(\p^i)}$.
\begin{proof}
	By our assumption on $\p^i$, we have that Lemma~\ref{lem:updown}~(1) fails, hence Lemma~\ref{lem:updown}~(2) fails so that $i\in\{2,\dots,S-1\}$ and $\up_\K(\p^i)\neq\down_\K(\p^i)$. It remains to show that the layering $\r$ has the desired property.

	Define the route $\r:=(\p^i\backslash\{p^i_{\up_\K(\p^i)},p^i_{\down_\K(\p^i)}\})\cup\{p^{i+1}_{\up_\K(\p^i)},p^{i-1}_{\down_\K(\p^i)}\}$.
	First note that $\r$ and $\p^i$ are not compatible, since 
	\[r_{\up_\K(\p^i)}=p^{i+1}_{\up_\K(\p^i)}\succ_{s_{\up_\K(\p^i)}}^+ p^i_{\up_\K(\p^i)}\text{ and }r_{\down_\K(\p^i)}=p^{i-1}_{\down_\K(\p^i)}\prec_{s_{\down_\K(\p^i)}}^+ p^i_{\down_\K(\p^i)}.\]

	We now show that $\r$ is compatible with every layering of $\K\backslash\{\p^i\}$.
	To this end, choose $j\in[S]$ distinct from $i$; we will show that $\r$ is compatible with $\p^j$.
	Since all routes of $\r$ are in layerings of $\K$, we have that $\{\r,\p^j\}$ are route-compatible.
	It remains to show that $\r\prec\p^j$ or $\p^j\prec\r$.

	Suppose first that $j<i$. For any $k\in[m]$ distinct from $\down_\K(\p^i)$, we have
	$p^j_k\preceq_{s_k}^+p^i_k\preceq r_k$.
	Moreover, if $k=\down_\K(\p^i)$ then
	$p^j_k\preceq_{s_k}^+p^{i-1}_k=r_k$.
	This shows that $\p^j\preceq\r$, hence $\r$ is compatible with $\p^j$. If $j>i$, then a symmetric argument shows that $\p^j\succeq\r$, so in this case $\r$ is again compatible with $\p^j$.
	This shows that $\L:=(\K\backslash\{\p^i\})\cup\{\r\}$ is a layering-clique
	with 
$\{\p^1\prec\dots\p^{i-1}\prec\r\prec\p^{i+1}\prec\dots\prec\p^S\}$.
	Since the cardinality of $\L$ is the same as that of $\K$, it is a maximal layering-clique by Corollary~\ref{cor:triangulation-cardinality}.
\end{proof}

Proposition~\ref{prop:shuffle} motivates the following definition.

\begin{defn}
	With the setup and notation of Proposition~\ref{prop:shuffle}, if $\up_\K(\p^i)<\down_\K(\p^i)$ then we say that $\L$ is a \emph{down-shuffle of $\K$ at $\p^i$} and we write
	\[\K\xrightarrow{\shuffle}\L.\]
	On the other hand, if $\up_\K(\p^i)>\down_\K(\p^i)$, then we call $\L$ an \emph{up-shuffle of $\K$ at $\p^i$} and we write $\K\xleftarrow{\shuffle}\L$.
	In either case, we call $\L$ a \emph{shuffle of $\K$ at $\p^i$}.
\end{defn}

Shuffling moves are one way of traveling from a maximal layering-clique of $\G$ to an adjacent maximal layering-clique by replacing the layering $\p^i$ with a different layering $\r$.
Note that in the notation of Proposition~\ref{prop:shuffle}, every route of $\r$ is in $\p^{i+1}$ or $\p^{i-1}$ with $\up_{\L}(\r)=\down_\K(\p^i)$ and $\down_{\L}(\r)=\up_\K(\p^i)$, hence performing a shuffle of $\L$ at $\r$ retrieves the original layering-clique $\K$.
This shows that a maximal layering-clique $\L$ is an up-shuffle of $\K$ if and only if $\K$ is a down-shuffle of $\L$.

In Figure~\ref{fig:three-Mutations}, the bottom-left layering-clique is a down-shuffle move of the top layering-clique at the magenta layering.

We now aim to describe two more ways to replace a layering $\p^i$ of a layering-clique $\K$ with a different layering to get a new maximal layering-clique. Both moves must take place in the setting when there exists some route of $\p^i$ which is in no other layering of $\K$, as the case where every route of $\p^i$ occurs in another layering of $\K$ is covered by Proposition~\ref{prop:shuffle}.
First, we need the following lemma.

\begin{lemma}\label{lem:hor-routes-maximal}
	Let $\K$ be a maximal layering-clique of $\G$. If $p$ is a horizontal route of $\G$ which is compatible with every route of $\routes(\K)$, then $p\in\routes(\K)$.
\end{lemma}
\begin{proof}
	Write $\K=\{\p^1,\dots,\p^S\}$.
	Choose $i\in[m]$ such that $p$ is a horizontal $i$-route.
	We will construct a layering containing $p$ which is compatible with every layering of $\K$.

	First, suppose that there exists a route of $\routes(\K)$ which is greater than or equal to $p$ in $\prec_{s_i}^+$. Then choose $q\in\routes(\K)$ such that $q\succeq_{s_i}^+p$ and, for any $r\in\routes(\K)$ with $r\succeq_{s_i}^+ p$, we have $r\succeq_{s_i}^+ q$ (i.e., $q$ covers $p$ in the restriction of $\prec_{s_i}^+$ to $\routes(\K)\cup\{p\}$).

	Choose $j\in[S]$ to be minimal such that $q=p^j_i$. Define the layering $\r$ such that \[r_b=\begin{cases}p^j_b&b\neq i\\p&b=i.\end{cases}\] We claim that $\r$ is compatible with every layering of $\K$.
		It is immediate by construction that $\K\cup\{\r\}$ is route-compatible, so it remains to show the second condition of compatibility.

	Indeed, choose any index $a\in[S]$. If $a\geq j$, then for any $b\in[m]$ we have
	\[p^a_b\succeq_{s_b}^+ p^j_b\succeq_{s_b}^+ r_b,\]
	so $\p^a\succ\r$ and these two layerings are compatible. 

	Now, suppose $a<j$. In particular, this implies that $j>1$. Note that by minimality of $j$, we have $p^{j-1}_i\prec_{s_i}^+p\preceq_{s_i}^+p^j_i$, so that $\down_\K(\p^j)=i$ and hence $p^j_b=p^{j-1}_b$ for all $b\neq i$. Then for any index $b\in[m]$ other than $i$ we have
	\[p^a_b\preceq_{s_i}^+ p^{j-1}_b=p^j_b.\]
	Moreover, $p^a_i\preceq_{s_i}^+p^{j-1}_i\prec_{s_i}^+ p=r_i$, completing the proof that $\p^a\prec\r$.

	We have shown that $\r$ is compatible with every layering of $\K$, hence $\r\in\K$ by maximality of $\K$. Since $p\in\r$, this completes the proof when there is a route of $\routes(\K)$ weakly above $p$ in $\prec_{s_i}^+$.

	If there is no route of $\routes(\K)$ weakly above $p$ in $\prec_{s_i}^+$, then there must be a route below $p$ in $\prec_{s_i}^+$ and a symmetric argument shows the result in this case.
\end{proof}

We are now able to understand how to relate adjacent maximal layering-cliques outside of the shuffle case with the following proposition.

\begin{prop}\label{prop:god}
	Let $\K=\{\p^1,\dots,\p^S\}$ and $\L=\{\q^1,\dots,\q^S\}$ be adjacent maximal layering-cliques and choose $i,j\in[S]$ so that $\L=(\K\backslash\{\p^i\})\cup\{\q^j\}$. Suppose there exists a route $p\in\p^i$ which is in no other layering of $\K$. Then there is a route $q\in\q^j$ incompatible with $p$ which is in no other layering of $\L$, so that
	\[\routes(\L)=(\routes(\K)\backslash\{p\})\cup\{q\}.\]
	Moreover,
	\begin{enumerate}
		\item If $i\in\{2,\dots,S-1\}$, then $i=j$ and $\updown_\K(\p^i)=\updown_\L(\q^i)$.
		\item If $i=1$, then $j=S$ and $\updown_\K(\p^1)=1+\updown_\L(\q^S)$.
		\item If $i=S$, then $j=1$ and $\updown_\L(\q^1)=1+\updown_\K(\p^S)$.
	\end{enumerate}
\end{prop}
\begin{proof}
	Note that the existence of distinct maximal layering-cliques necessitates $\dim\F_\G(1)\geq1$ by Lemma~\ref{lem:dim0} and hence $S\geq2$ by Remark~\ref{remk:geq1}.

	We first show that there is a route of $\q^j$ which is in no other layering of $\L$.
	Suppose to the contrary that every route of $\q^j$ is in a different layering of $\L$.
	Then every route of $\q^j$ is in a layering of $\K\cap\L=\K\backslash\{\p^i\}$. Then $\routes(\L)=\routes(\K)\backslash\{p\}$, contradicting Lemma~\ref{lem:hor-routes-maximal}. It follows that there exists a (necessarily unique) route $q:=q^j_{\updown_\L(\q^j)}$ of $\q^j$ which is in no other layering of $\L$.

	For notational brevity, we write $u:=\updown_\K(\p^i)$ and $v:=\updown_\L(\q^j)$ for the rest of this proof. Then we have $p=p^i_u$ and $q=q^j_v$.

	We now show that $q^j_v\neq p^i_u$. Suppose for contradiction that $q^j_v=p^i_u$ (in particular, this implies that $u=v$). We first show that in this case we must have $i=j$.
	For any $k\in[S]$ such that $k<i$, note that $\p^k\prec\p^i$ implies that $p^k_u\prec_{s_u}^+p^i_u$. The layering $\p^k$ is also in $\L$, and $p^k_u\prec_{s_u}^+ p^i_u=q^j_u$ implies that $\p^k\prec\q^j$. Since this holds for $k\in\{1,\dots,i-1\}$, this shows that $j\geq i$.
	Similarly, if $k>i$ then $\p^k\succ\p^i$ implies that $p^k_u\succ_{s_u}^+ p^i_u$. The layering $\p^k$ is also in $\L$, and $p^k_u\succ_{s_u}^+ p^i_u=q^j_u$ implies that $\p^k\succ\q^j$, showing that $j\leq i$ and hence $j=i$.
	Then the ordering $\prec_\L$ is
	\[\p^1\prec\p^2\prec\dots\prec\p^{i-1}\prec\q^j\prec\p^{i+1}\dots\prec\p^S.\]
	{We finally split into the case when $i<S$ and $i=S$. If $i<S$, then Lemma~\ref{lem:b} shows that $q^j_b=\begin{cases}p^{i+1}_b&b\neq u\\q^j_u&b=u\end{cases}=p^i_b$ and hence $\q^j=\p^i$. On the other hand, if $i=S$ then Lemma~\ref{lem:b} shows that $q^j_b=\begin{cases}p^{i-1}_b&b\neq u\\q^j_u&b=u\end{cases}=p^i_b$ and hence $\q^j=\p^i$. In either case, we have $\q^j=\p^i$, contradicting the assumption that $\K$ and $\L$ are distinct maximal layering-cliques.}
		We have shown that $q^j_v\neq p^i_u$. 
	It finally follows that $\routes(\L)=(\routes(\K)\backslash\{p^i_u\})\cup\{q^j_v\}$.

	It remains to show the final three parts of the proposition statement. 
	Since $q^j_v$ is not in $\routes(\K)$, Lemma~\ref{lem:hor-routes-maximal} shows that $q^j_v$ must be incompatible with a route of $\routes(\K)$. Since $p^i_u$ is the only route of $\routes(\K)\backslash\routes(\L)$, it must be the case that $q^j_v$ and $p^i_u$ are incompatible as routes.

	First, suppose that $i\in\{2,\dots,S-1\}$. We first show that $u=v$. Suppose to the contrary that $u<v$. The route $p^{i+1}_u$ is in both $\routes(\L)$ and $\routes(\K)$ and satisfies $p^{i+1}_u\succ_{s_i}^+p^i_u$, hence $p^{i+1}_u$ lies weakly above $p^i_u$ in the planar embedding of $\G$. The route $q^j_v$ starts at the source $s_v$ (above $s_u$) and is compatible with $p^{i+1}_u$, hence $q^j_v$ lies weakly above $p^{i+1}_u$ in the embedding of $\G$. This means that $q^j_v$ lies weakly above $p^{i}_u$ in the embedding of $\G$, contradicting the fact that $q^j_v$ and $p^i_u$ are incompatible. This shows that we cannot have $u<v$. Symmetrically, we cannot have $u>v$; it follows that $u=v$.

	Similarly, we now show that $i=j$. Suppose to the contrary that $i<j$. 
	Then $\p^a=\q^a$ for $a<i$ and $\p^{i+1}=\q^i$.
	Since $j>i$, we have $q^{j}_u\succ_{s_u}^+ q^i_u=p^{i+1}_u$, so $q^j_u$ lies weakly above $p^{i+1}_u$. As in the above paragraph, we have that $p^i_u$ lies weakly below $p^{i+1}_u$. Then $p^i_u$ lies weakly below $q^j_u$, contradicting that these two routes are incompatible.
	It follows that $i\geq j$. Symmetrically, we must have $i\leq j$ and hence $i=j$.
	This completes the proof of (1) covering the case when $i\in\{2,\dots,S-1\}$.

	We now prove (2). Suppose that $i=1$.
	We first show that $u\geq v$. Suppose to the contrary that $u<v$.
	The route $p^{2}_u$ is in both $\routes(\L)$ and $\routes(\K)$ and satisfies $p^{2}_u\succ_{s_1}^+p^1_u$, hence $p^{2}_u$ lies weakly above $p^1_u$ in the planar embedding of $\G$. The route $q^j_v$ starts at the source $s_v$ (above $s_u$) and is compatible with $p^{2}_u$, hence $q^j_v$ lies weakly above $p^{2}_u$ in the embedding of $\G$. This means that $q^j_v$ lies weakly above $p^{1}_u$ in the embedding of $\G$, contradicting the fact that $q^j_v$ and $p^1_u=p^i_u$ are incompatible. This shows that we cannot have $u<v$.

	Suppose now that $u=v$. 
	We show that we must have $1=i=j$ as in the $u\in\{2,\dots,S-1\}$ case.
	Suppose to the contrary that $j>1$.
	Then $\p^{2}=\q^1$.
	Since $j>1$, we have $q^{j}_u\succ_{s_u}^+ q^1_u=p^{2}_u$, so $q^j_u$ lies weakly above $p^{2}_u$. On the other hand, $p^1_u$ lies weakly below $p^{2}_u$. Then $p^1_u$ lies weakly below $q^j_u$, contradicting that these two routes are incompatible.
	It follows that if $u=v$, then $1=i=j$.

	Suppose without loss of generality that $p^1_u\succ_{s_u}^+q^1_u$ (if not, then we may switch $\K$ and $\L$ since $i=j=1$ and $u=v$ ensure that the setup is symmetric).
	Since $p^1_u$ and $q^1_u$ are incompatible, they must have an incompatibility; let $\sigma$ be the incompatibility between these routes which occurs as early as possible along $p^1_u$ and $q^1_u$. Then we may write
	\[p^1_u=p^-\sigma p^+ \text{ and }q^1_u=q^-\sigma q^+\]
	where, for example, the path $p^-$ is the subpath of $p^1_u$ occurring before $\sigma$.
	Since $p^1_u\succ_{s_u}^+q^1_u$ and $\sigma$ occurs as early as possible, it must be the case that $p^-$ enters $t(\sigma)$ above $q^-$, and $p^+$ leaves $h(\sigma)$ below $q^+$.
	Define the route $r:=q^-\sigma p^+$.
	Note immediately that $r$ is weakly below $p^1_u$ and weakly below $q^1_u$ in the planar embedding, so it is compatible with both of these routes.
	Moreover, the route $r$ is not in $\routes(\K)$, because $p^1_u\in\p^1$ means that $p^1_u$ is the $\prec_{s_u}^+$-lowest route of $\routes(\K)$ starting at $s_u$ while $r\prec_{s_u}^+p^1_u$.

	If $u=v=1$, then every route of $\routes(\K)$ is weakly above $p^1_u$ in the planar embedding of $\G$, hence is weakly above $r$. This implies that $r$ is compatible with every route of $\routes(\K)$. Then $r\in\routes(\K)$ by Lemma~\ref{lem:hor-routes-maximal}, a contradiction. On the other hand, suppose $u=v>1$. Consider the route $p^S_{u-1}$. This route is weakly below both $p^1_u$ and $q^1_u$ in the planar embedding, hence it is weakly below $r$. Moreover, every route $p'$ of $\routes(\K)$ is either weakly above $p^1_u$ or weakly below $p^S_{u-1}$; in either case, $p'$ is compatible with $r$. Again, this shows that $r$ is compatible with every route of $\routes(\K)$, hence $r\in\routes(\K)$ by Lemma~\ref{lem:hor-routes-maximal}, a contradiction. This completes our contradiction when $u=v$, proving that we must have $u>v$.

	If $u>v+1$ and/or $j<S$, then the route $q^j_v$ is weakly below $p^S_{u-1}$ and hence weakly below $p^1_u$, contradicting that $q^j_v$ and $p^1_u$ are incompatible. It finally follows that we must have $v=u-1$ and $j=S$, finishing the proof of (2).

	The proof of (3) is symmetric to the proof of (2).
\end{proof}

We now finally define rotations and shuffles.

\begin{defn}
	With the setup and notation of Proposition~\ref{prop:god}, if $i\in\{2,\dots,S-1\}$ and $p^i_{\updown_\K(\p^i)}\succ_{s_i}^+ q^i_{\updown_\K(\p^i)}$ then we say that $\L$ is a \emph{down-rotation of $\K$ at $\p^i$} and we write $\K\xrightarrow{\rotation}\L$. If $i\in\{2,\dots,S-1\}$ and $p^i_{\updown_\K(\p^i)}\prec_{s_i}^+ q^i_{\updown_\K(\p^i)}$ then we say that $\L$ is an \emph{up-rotation of $\K$ at $\p^i$} and we write $\K\xleftarrow{\rotation}\L$. In either case, we call $\L$ a \emph{rotation of $\K$ at $\p^i$}.
\end{defn}

\begin{defn}
	With the setup and notation of Proposition~\ref{prop:god}, if $i=S$ then we say that $\L$ is a \emph{down-realignment of $\K$ at $\p^i$} and we write $\K\xrightarrow{\realignment}\L$. If $i=1$ then we say that $\L$ is an \emph{up-realignment of $\K$ at $\p^i$} and we write $\K\xleftarrow{\realignment}\L$. In either case, we call $\L$ a \emph{realignment of $\K$ at $\p^i$}.
\end{defn}

It is an immediate consequence of Proposition~\ref{prop:god} that $\L$ is a down-rotation of $\K$ if and only if $\K$ is an up-rotation of $\L$, and that $\L$ is a down-realignment of $\K$ if and only if $\K$ is an up-realignment of $\L$.

In Figure~\ref{fig:three-Mutations}, the bottom-right layering-clique is the down-realignment of the top layering-clique at the orange layering. The bottom-middle layering-clique is the down-rotation of the top layering-clique at the green layering.

\begin{figure}
	\centering
	\def\svgscale{.38}
\begingroup%
  \makeatletter%
  \providecommand\color[2][]{%
    \errmessage{(Inkscape) Color is used for the text in Inkscape, but the package 'color.sty' is not loaded}%
    \renewcommand\color[2][]{}%
  }%
  \providecommand\transparent[1]{%
    \errmessage{(Inkscape) Transparency is used (non-zero) for the text in Inkscape, but the package 'transparent.sty' is not loaded}%
    \renewcommand\transparent[1]{}%
  }%
  \providecommand\rotatebox[2]{#2}%
  \newcommand*\fsize{\dimexpr\f@size pt\relax}%
  \newcommand*\lineheight[1]{\fontsize{\fsize}{#1\fsize}\selectfont}%
  \ifx\svgwidth\undefined%
    \setlength{\unitlength}{1044.47998047bp}%
    \ifx\svgscale\undefined%
      \relax%
    \else%
      \setlength{\unitlength}{\unitlength * \real{\svgscale}}%
    \fi%
  \else%
    \setlength{\unitlength}{\svgwidth}%
  \fi%
  \global\let\svgwidth\undefined%
  \global\let\svgscale\undefined%
  \makeatother%
  \begin{picture}(1,0.57108035)%
    \lineheight{1}%
    \setlength\tabcolsep{0pt}%
    \put(0,0){\includegraphics[width=\unitlength,page=1]{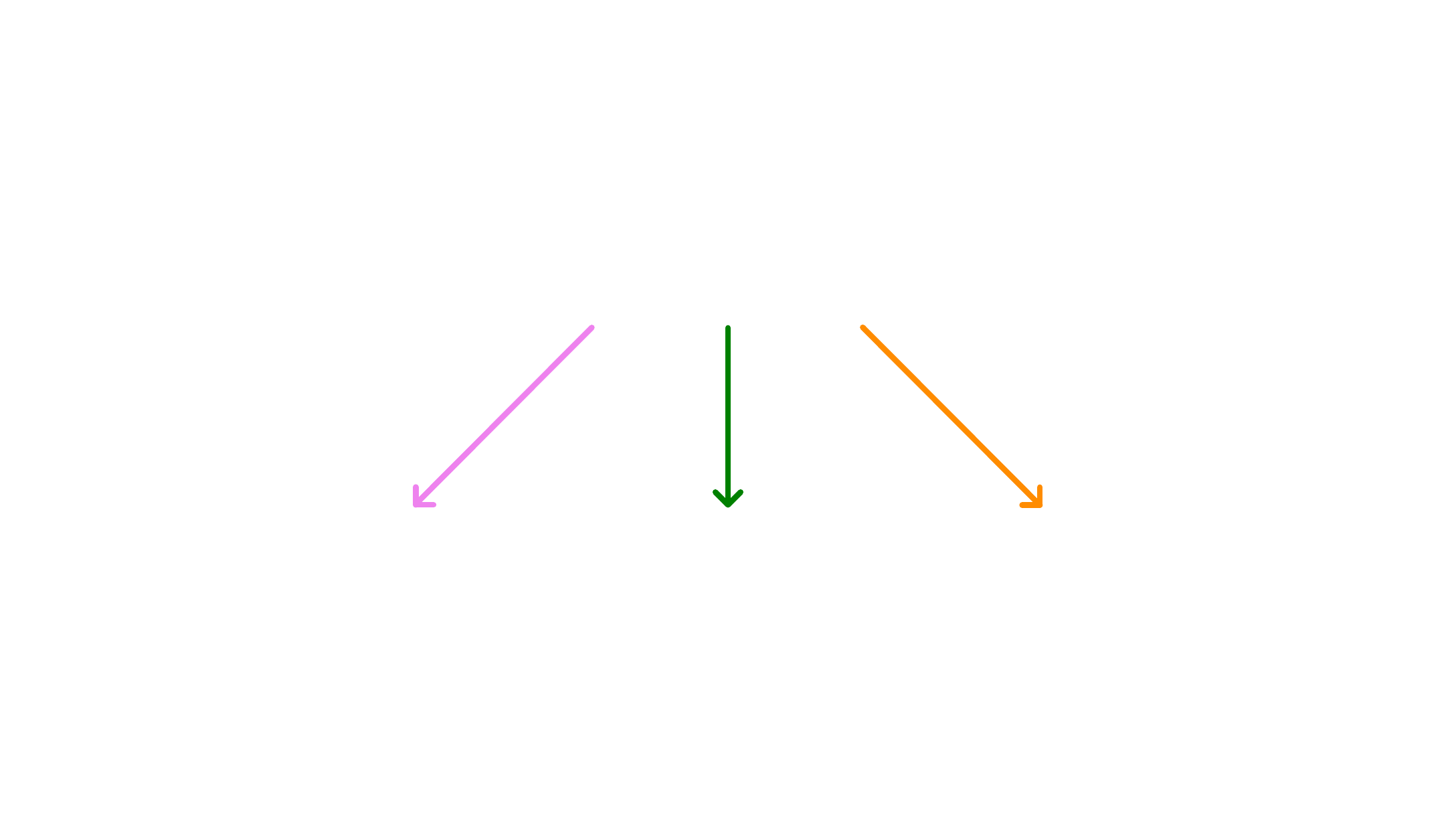}}%
    \put(0.31921483,0.28890548){\color[rgb]{0.93333333,0.50980392,0.93333333}\makebox(0,0)[lt]{\lineheight{1.25}\smash{\begin{tabular}[t]{l}$\sh$\end{tabular}}}}%
    \put(0.47284947,0.28890548){\color[rgb]{0,0.50196078,0}\makebox(0,0)[lt]{\lineheight{1.25}\smash{\begin{tabular}[t]{l}$\ro$\end{tabular}}}}%
    \put(0.67245042,0.28890548){\color[rgb]{1,0.54901961,0}\makebox(0,0)[lt]{\lineheight{1.25}\smash{\begin{tabular}[t]{l}$\re$\end{tabular}}}}%
    \put(0,0){\includegraphics[width=\unitlength,page=2]{im_threemutations.pdf}}%
  \end{picture}%
\endgroup%

	\caption{From the top maximal layering-clique is shown a down-shuffle, down-rotation, and down-realignment.}
	\label{fig:three-Mutations}
\end{figure}

\begin{thm}
	\label{thm:shuffle-rotation-realignment}
	Let $\K$ and $\L$ be adjacent maximal layering-cliques. Then $\L$ is a shuffle, rotation, or realignment of $\K$.
\end{thm}
\begin{proof}
	Write $\K=\{\p^1,\dots,\p^S\}$ and $\L=\{\q^1,\dots,\q^S\}$. Since $\K$ and $\L$ are adjacent, there exist indices $i,j\in[S]$ such that $\L=(\K\backslash\{\p^i\})\cup\{\q^j\}$.
	If every route of $\p^i$ is in another layering of $\K$, then Proposition~\ref{prop:shuffle} gives that $\L$ is a shuffle of $\K$. Otherwise, if $i=1$ or $i=S$ then $\L$ is a realignment of $\K$, and if $i\in\{2,\dots,S-2\}$ then $\L$ is a rotation of $\K$.
\end{proof}

We refer to shuffle, rotation, and realignment moves collectively as \emph{mutations}. If $\L$ is a shuffle, rotation, or realignment of $\K$ at a layering $\p\in\K$, we say that $\L$ is a \emph{mutation of $\K$ at $\p$} and write $\L=\mu_\p(\K)$. If $\L$ is an down-shuffle, down-rotation, or down-realignment of $\K$ at a layering $\p\in\K$, we say that $\L$ is a \emph{down-mutation} of $\K$ at $\p$. We similarly refer to \emph{up-mutations} of $\K$.

\begin{remk}\label{remk:poly2}
	Recall from Remark~\ref{remk:poly1} that up to integral equivalence $\F_\G(1)$ is the restriction of $\F_{\hat\G}(m\times1)$ to the hyperplanes setting flow through all extension edges to 1. Dilating the (unimodular) framing triangulation of $\F_{\hat\G}(1)$ to a (non-unimodular) triangulation of $\F_{\hat\G}(m\times1)$ and restricting to $\F_\G(1)$ gives a subdivision $\mathcal S$ of
	$\F_\G(1)$ wherein two flows $F_1$ and $F_2$ of $\F_\G(1)$ are in the same maximal cell if and only if $\routes(F_1)\cup\routes(F_2)$ is a clique, where $\routes(F_1)$ is the set of routes used in the clique combination for $F_1$.
	Because $\routes(K)$ is a clique for any layering $K$ of $\G$, the framing triangulation of $\F_\G(1)$ refines the subdivision $\mathcal S$.
	In general this refinement is strict because if $\K\xrightarrow{\shuffle}\L$ is a shuffle move on maximal layering-cliques then $\routes(\K)=\routes(\L)$ -- i.e., maximal simplices of the framing triangulation related by shuffle moves are contained in the same maximal cell of $\mathcal S$.
	On the other hand, if $\K\to\L$ is a rotation or realignment then $\routes(\K)\cup\routes(\L)$ is not a clique and hence the simplices $\D_1(\K)$ and $\D_1(\L)$ are in different maximal cells of $\mathcal S$.
	Hence, the maximal cells of this subdivision $\mathcal S$ are obtained as unions of maximal simplices of the framing triangulation related by shuffle moves.
\end{remk}

\section{The Framing Poset}
\label{sec:poset}

We now define the framing poset, which is a partial order on maximal layering-cliques of a strongly planar balanced DAG $\G$ (and hence, simplices of its framing triangulation). 

\subsection{The post-source order on maximal layering-cliques}

We will define the framing poset as the transitive closure of down-mutations, but it is not \textit{a priori} clear that this definition is antisymmetric.
To show antisymmetry, we will use a simpler lexicographic order on maximal layering-cliques which we define in this section.

First, we need a total order on layerings:

\begin{defn}
	Let $\p$ and $\q$ be distinct layerings of $\G$. Let $i\in[m]$ be maximal such that $p_i\neq q_i$. If $p_i\prec_{s_i}^+ q_i$, then we say that $\p\prec_{s}^+\q$. Otherwise, $p_i\succ_{s_i}^+q_i$ and we say that $\p\succ_s^+\q$.
	We say that $\succ_s^+$ is the \emph{post-source order} on layerings of $\G$.
\end{defn}
The post-source order $\succ_s^+$ is a lexicographic order on layerings as ordered tuples of routes under $\prec_{s_j}^+$, so it gives a total order on layerings of $\G$.
\begin{defn}
	Let $\K=\{\p^1,\dots,\p^S\}$ and $\L=\{\q^1,\dots,\q^S\}$ be distinct maximal layering-cliques of $\G$. Let $j\in[S]$ be maximal such that $\p^j\neq\q^j$.
	If $\p^j\prec_s^+\q^j$, then we say that $\K\prec_s^+\L$; otherwise, $\p^j\succ_s^+\q^j$ and we say that $\K\succ_s^+\L$.
	We say that $\succ_s^+$ is the \emph{post-source order} on maximal layering-cliques of $\G$.
\end{defn}
The post-source order $\succ_s^+$ is a lexicographic order on maximal layering-cliques as ordered tuples of layerings under the post-source order $\prec_s^+$ on layerings, hence it gives a total order on maximal layering-cliques of $\G$.
We now show that up-mutations and down-mutations interact well with the post-source order on maximal layering-cliques.

\begin{prop}
	\label{prop:mutation-down}
	If a maximal layering-clique $\L$ is a down-mutation of $\K$, then $\L\prec_s^+\K$.
\end{prop}
\begin{proof}
	Write $\K=\{\p^1,\dots,\p^S\}$ and $\L=\{\q^1,\dots,\q^S\}$. Since $\L$ is a down-mutation of $\K$, these layering-cliques are adjacent and hence $(\K\backslash\{\p^i\})\cup\{\q^j\}=\L$ for some $i,j\in[m]$.

	Suppose first that $\L$ is a down-shuffle of $\K$. Then
	by Proposition~\ref{prop:shuffle}, we have $i\in\{2,\dots,S-1\}$ and $\q^j=(\p^i\backslash\{p^i_{\up_\K(\p^i)},p^i_{\down_\K(\p^i)}\})\cup\{p^{i+1}_{\up_\K(\p^i)},p^{i-1}_{\down_\K(\p^i)}\}$ and
	\[\p^1\prec\dots\prec\p^{i-1}\prec\q^j\prec\p^{i+1}\prec\dots\prec\p^S\]
	is the layering-clique $\L$ ordered by $\prec$, showing that $j=i$.
	It follows that $i$ is the maximal (indeed, the only) index in $[S]$ such that $\p^i\neq\q^i$. Since $\L$ is a down-shuffle of $\K$, we have $\up_\K(\p^i)<\down_\K(\p^i)$ so that the index $u:=\down_\K(\p^i)$ is the largest index in $[m]$ such that $p^i_u$ and $q^i_u$ differ; since $p^i_u\succ_{s_u}^+p^{i-1}_u=q^{i}_u$ we then have $\p^i\succ_s^+\q^i$, hence $\K\succ_s^+\L$, proving the proposition when $\L$ is a down-shuffle of $\K$.

	Suppose now that $\L$ is a down-rotation of $\K$. By Proposition~\ref{prop:god}, we have $i=j$, so that $i$ is the unique index such that the layerings $\p^i$ and $\q^i$ differ. Also by Proposition~\ref{prop:god}, the index $\updown_\K(\p^i)=\updown_\L(\q^i)=:u$ is the unique index such that the routes $p^i_u$ and $q^i_u$ differ. Since $p^i_u\succ_{s_i}^+q^i_u$, we have $\p^i\succ\q^j$ and hence $\K\succ_s^+\L$, proving the proposition when $\L$ is a down-rotation of $\K$.

	Suppose finally that $\L$ is a down-realignment of $\K$. Then $i=S$ and $j=1$. In particular, this means that $\p^S\succ_{s}^+\p^{S-1}=\q^S$, hence $\K\succ_{s}^+\L$. This ends the proof.
\end{proof}

\subsection{The framing poset on maximal layering-cliques}

We are now able to define the framing poset on maximal layering-cliques.

\begin{defn}
	If a maximal layering-clique $\L$ is a down-mutation of a maximal layering-clique $\K$, then we write that $\L\prec\K$. Let $\prec$ be the transitive closure of these relations.
\end{defn}

In other words, we write $\L\preceq\K$ if there is a (possibly empty) series of down-mutations starting at $\K$ and ending at $\L$.

\begin{prop}\label{prop:partial-order}
	The relation $\preceq$ is a partial order on the set of maximal layering-cliques of $\G$.
\end{prop}
\begin{proof}
	If $\dim\F_\G(1)=0$, then Lemma~\ref{lem:dim0} gives that there is only one maximal layering-clique of $\G$ and hence the result is trivial. Now suppose $\dim\F_\G(1)\geq1$.

	It is immediate by definition that $\preceq$ is reflexive and transitive. It remains to show that $\preceq$ is antisymmetric.
	Suppose to the contrary that there exist distinct maximal layering-cliques $\K$ and $\L$ such that $\K\preceq\L$ and $\L\preceq\K$.
	Since $\L\preceq\K$, there must exist a sequence of down-mutations
	\[\K\to\K_1\to\dots\to\K_S=\L.\]
	Then Proposition~\ref{prop:mutation-down} shows that $\L\preceq_s^+\K$. One may show symmetrically that $\K\preceq_s^+\L$, contradicting the fact that $\preceq_s^+$ is a total order.
	It follows that $\preceq$ is antisymmetric, and hence that $\preceq$ is a poset.
\end{proof}

We call $\preceq$ the \emph{framing poset} of $\G$.
See Figure~\ref{fig:small_example}, where the maximal layering-cliques are arranged into the Hasse diagram of the framing poset on the right.

\subsection{Relations to one-source-one-sink framing triangulations}

Recall the definition of a planar-framed DAG in Section~\ref{ssec:planar-framing}.
In particular, if a strongly planar balanced DAG $\G$ has one source and one sink, then the embedding of $\G$ induces a planar framing $\R_\P$ of $\G$.
In this case, a layering is merely a route and our compatibility of layerings amounts to compatibility of routes as defined in~\cite{DKK}, so our framing triangulations induced by the embedding $\G$ are the same as the framing triangulations induced by the planar framing $\R_\P$.
Moreover, shuffles and realignments are impossible in this case and the definition of our framing poset agrees with that given in~\cite{vBC}, hence we recover the framing triangulations and framing lattices of planar-framed DAGs with one source and one sink. This means that the framing triangulation and framing lattice for the unit flow polytope of any planar-framed DAG with one source and one sink as defined in the literature are the same as its framing triangulation and framing poset in our sense -- for example, see the DAG of Figure~\ref{fig:square}.

\begin{remk}
	We remark that~\cite{vBC} takes the convention when drawing their framed DAGs that, for example, if $\alpha$ and $\beta$ are two edges incoming to the same vertex $v$ and $\alpha$ is drawn above $\beta$ in the embedding, then $\alpha<_{\R,v}\beta$. This is opposite to our convention. In particular, this means that their ``clockwise rotations'' will look like counter-clockwise rotations in our embeddings, and vice versa.
\end{remk}

\subsection{Examples}

We now give three more examples of framing posets of strongly planar balanced DAGs outside of the unit one-source-one-sink case.

\begin{example}
	Let $\G$ be the strongly planar balanced DAG of Figure~\ref{fig:zigzag}. Shown on the right are its four maximal layering-cliques arranged into the framing poset of $\G$. In particular, note that the framing poset is not a lattice because it has no unique greatest element and no unique least element.
	\begin{figure}
		\centering
		\def\svgscale{.28}
\begingroup%
  \makeatletter%
  \providecommand\color[2][]{%
    \errmessage{(Inkscape) Color is used for the text in Inkscape, but the package 'color.sty' is not loaded}%
    \renewcommand\color[2][]{}%
  }%
  \providecommand\transparent[1]{%
    \errmessage{(Inkscape) Transparency is used (non-zero) for the text in Inkscape, but the package 'transparent.sty' is not loaded}%
    \renewcommand\transparent[1]{}%
  }%
  \providecommand\rotatebox[2]{#2}%
  \newcommand*\fsize{\dimexpr\f@size pt\relax}%
  \newcommand*\lineheight[1]{\fontsize{\fsize}{#1\fsize}\selectfont}%
  \ifx\svgwidth\undefined%
    \setlength{\unitlength}{1587.35913086bp}%
    \ifx\svgscale\undefined%
      \relax%
    \else%
      \setlength{\unitlength}{\unitlength * \real{\svgscale}}%
    \fi%
  \else%
    \setlength{\unitlength}{\svgwidth}%
  \fi%
  \global\let\svgwidth\undefined%
  \global\let\svgscale\undefined%
  \makeatother%
  \begin{picture}(1,0.32452706)%
    \lineheight{1}%
    \setlength\tabcolsep{0pt}%
    \put(0,0){\includegraphics[width=\unitlength,page=1]{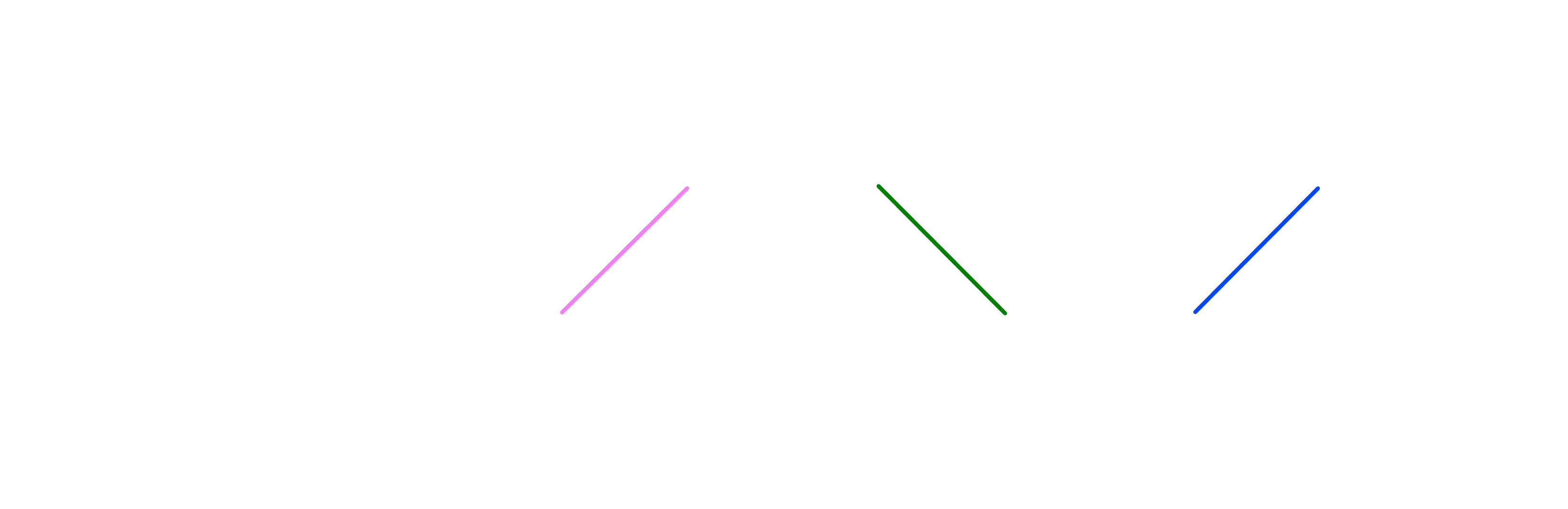}}%
    \put(0.77042564,0.16075442){\color[rgb]{0.00784314,0.27843137,0.99607843}\makebox(0,0)[lt]{\lineheight{1.25}\smash{\begin{tabular}[t]{l}$\sh$\end{tabular}}}}%
    \put(0.61508565,0.1617271){\color[rgb]{0,0.50196078,0}\makebox(0,0)[lt]{\lineheight{1.25}\smash{\begin{tabular}[t]{l}$\re$\end{tabular}}}}%
    \put(0,0){\includegraphics[width=\unitlength,page=2]{im_zigzag.pdf}}%
    \put(0.36724028,0.16075442){\color[rgb]{0.93333333,0.50980392,0.93333333}\makebox(0,0)[lt]{\lineheight{1.25}\smash{\begin{tabular}[t]{l}$\sh$\end{tabular}}}}%
    \put(0,0){\includegraphics[width=\unitlength,page=3]{im_zigzag.pdf}}%
  \end{picture}%
\endgroup%

		\caption{A DAG whose framing poset is not a lattice.}
		\label{fig:zigzag}
	\end{figure}
\end{example}

\begin{example}\label{ex:nonunit}
	On the top-left of Figure~\ref{fig:blowup} is a strongly planar DAG $\G$ with netflow vector $\a=(2,-2)$. Note that this flow polytope is not unit, so we do not consider layerings on $(\G,\a)$ directly. Rather, we decontract to the strongly planar balanced DAG $\G'$ shown on the top-right (recall Proposition~\ref{prop:reduce_to_balanced}) and observe that $\F_\G(\a)\ucong\F_{G'}(1)$. Shown finally are the two maximal layering-cliques of $\G'$; these index the triangulation of the line segment $\F_{G'}(1)$ into two unimodular one-dimensional simplices, shown on the bottom-left of the figure.
	This in turn triangulates $\F_\G(\a)$; in this way, we can triangulate even non-unit flow polytopes of strongly planar pairs $(\G,\a)$ using techniques of this paper.
\begin{figure}
	\centering
	\def\svgscale{.28}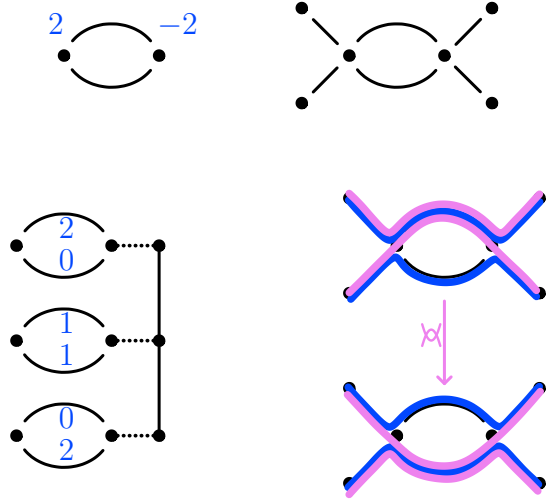
	\caption{A DAG of a non-unit flow polytope, its reduction to a unit flow polytope, and the framing poset.}
	\label{fig:blowup}
\end{figure}
	We finally remark that, more generally, the DAG $\G$ of Figure~\ref{fig:blowup} with netflow $(n,-n)$ for any $n\in\mathbb N_{\geq0}$ gives a flow polytope which is (integrally equivalent to) a line segment of length $n$ triangulated into unit segments.
\end{example}

\begin{example}
	Consider the strongly planar balanced DAG $\G$ of Figure~\ref{fig:shuffles}, which consists of three connected components. Its framing poset is the hexagon shown on the right of the figure, all of whose edges are shuffles. Note that the flow polytope $\F_\G(1)$ is the Minkowski sum of the unit flow polytopes of the connected components of $\G$, but the choice of planar embedding of $\G$ induces an ordering of its connected components from bottom to top, which is used to construct the framing triangulation.
\begin{figure}
	\centering
	\def\svgscale{.19}
\begingroup%
  \makeatletter%
  \providecommand\color[2][]{%
    \errmessage{(Inkscape) Color is used for the text in Inkscape, but the package 'color.sty' is not loaded}%
    \renewcommand\color[2][]{}%
  }%
  \providecommand\transparent[1]{%
    \errmessage{(Inkscape) Transparency is used (non-zero) for the text in Inkscape, but the package 'transparent.sty' is not loaded}%
    \renewcommand\transparent[1]{}%
  }%
  \providecommand\rotatebox[2]{#2}%
  \newcommand*\fsize{\dimexpr\f@size pt\relax}%
  \newcommand*\lineheight[1]{\fontsize{\fsize}{#1\fsize}\selectfont}%
  \ifx\svgwidth\undefined%
    \setlength{\unitlength}{1293.9329834bp}%
    \ifx\svgscale\undefined%
      \relax%
    \else%
      \setlength{\unitlength}{\unitlength * \real{\svgscale}}%
    \fi%
  \else%
    \setlength{\unitlength}{\svgwidth}%
  \fi%
  \global\let\svgwidth\undefined%
  \global\let\svgscale\undefined%
  \makeatother%
  \begin{picture}(1,1.04425118)%
    \lineheight{1}%
    \setlength\tabcolsep{0pt}%
    \put(0,0){\includegraphics[width=\unitlength,page=1]{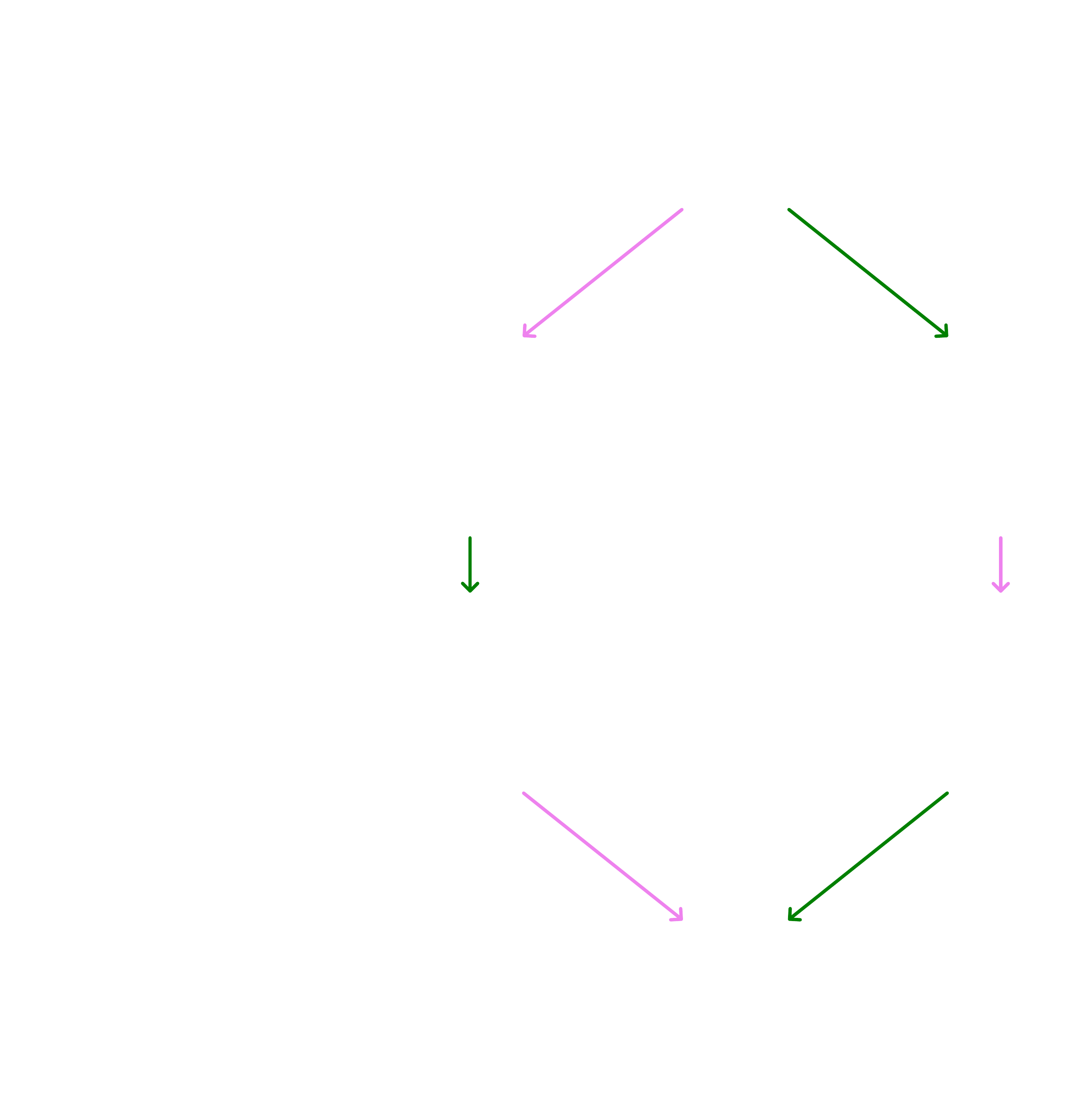}}%
    \put(0.88577988,0.50422244){\color[rgb]{0.93333333,0.50980392,0.93333333}\makebox(0,0)[lt]{\lineheight{1.25}\smash{\begin{tabular}[t]{l}$\sh$\end{tabular}}}}%
    \put(0.39116369,0.50422321){\color[rgb]{0,0.50196078,0}\makebox(0,0)[lt]{\lineheight{1.25}\smash{\begin{tabular}[t]{l}$\sh$\end{tabular}}}}%
    \put(0.8323554,0.79274974){\color[rgb]{0,0.50196078,0}\makebox(0,0)[lt]{\lineheight{1.25}\smash{\begin{tabular}[t]{l}$\sh$\end{tabular}}}}%
    \put(0.52306132,0.79274974){\color[rgb]{0.93333333,0.50980392,0.93333333}\makebox(0,0)[lt]{\lineheight{1.25}\smash{\begin{tabular}[t]{l}$\sh$\end{tabular}}}}%
    \put(0.5848697,0.24867211){\color[rgb]{0.93333333,0.50980392,0.93333333}\makebox(0,0)[lt]{\lineheight{1.25}\smash{\begin{tabular}[t]{l}$\sh$\end{tabular}}}}%
    \put(0.77036934,0.24867211){\color[rgb]{0,0.50196078,0}\makebox(0,0)[lt]{\lineheight{1.25}\smash{\begin{tabular}[t]{l}$\sh$\end{tabular}}}}%
    \put(0,0){\includegraphics[width=\unitlength,page=2]{im_shuffles.pdf}}%
  \end{picture}%
\endgroup%

	\caption{The framing poset of the unit flow polytope of a non-connected strongly planar balanced DAG.}
	\label{fig:shuffles}
\end{figure}

	Note that this framing poset is isomorphic to the weak order on $\mathfrak{S}_3$. 
	In fact, adding more connected components to the graph $\G$ gives larger weak orders through the framing poset.
	More generally, for $n\in\mathbb N_{\geq1}$ let $\Delta_n$ be the strongly planar DAG consisting of one source, one sink, and $n$ edges from source to sink. For a sequence of positive integers $\mathbf{s}=(s_1,\dots,s_m)$, let $\G_{(n_1,\dots,n_m)}$ be the strongly planar DAG (so that $\G$ of Figure~\ref{fig:shuffles} is $\G_{(2,2,2)}$).
	Then the framing poset of $\G_{\mathbf{s}}$ is the lattice of multipermutations $\mathcal{M}_{\mathbf{s}}$ as in~\cite[\S2.3]{vBC}.
\end{example}

\section{Future Directions and Obstacles}
\label{sec:obs}

We begin this section by presenting an example of a reasonably small balanced DAG whose unit flow polytope admits no triangulation induced by a pairwise compatibility condition on its integer points.
Finally, we finish with some conjectures and future directions.

\subsection{An inconvenient example}

It is natural to wonder if results about framing triangulations and framing posets may be extended beyond the strongly planar setting of this paper and the one-sink settings of~\cite{vBC,DHY}.
These existing theories of framing triangulations all work by defining a pairwise compatibility condition on the integer points of a flow polytope and defining the simplices of a triangulation to be the maximal collections of pairwise compatible points; Example~\ref{ex:k33} gives a balanced DAG whose unit flow polytope cannot be triangulated in this way.
This indicates that, while it may be possible to find a common generalization of framing triangulation theory of this paper and~\cite{vBC,DHY} using pairwise compatibility on integer points, it will be necessary to move away from pairwise compatibility conditions to get any sort of framing triangulation theory for arbitrary flow polytopes.

\begin{example}\label{ex:k33}
	Let $\G$ be the complete bipartite graph $\mathcal K_{3,3}$ shown in Figure~\ref{fig:no_triangulation} with edges oriented left to right.
	Label its source vertices $s_1,s_2,s_3$ and label its sink vertices $t_1,t_2,t_3$. For $i,j\in[3]$, write $\gamma_{i,j}$ to refer to the edge of $\G$ from $s_i$ to $t_j$.
	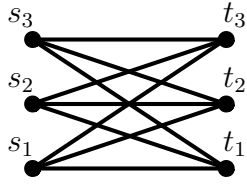
\begin{figure}
		\centering
		\def\svgscale{.38}
\begingroup%
  \makeatletter%
  \providecommand\color[2][]{%
    \errmessage{(Inkscape) Color is used for the text in Inkscape, but the package 'color.sty' is not loaded}%
    \renewcommand\color[2][]{}%
  }%
  \providecommand\transparent[1]{%
    \errmessage{(Inkscape) Transparency is used (non-zero) for the text in Inkscape, but the package 'transparent.sty' is not loaded}%
    \renewcommand\transparent[1]{}%
  }%
  \providecommand\rotatebox[2]{#2}%
  \newcommand*\fsize{\dimexpr\f@size pt\relax}%
  \newcommand*\lineheight[1]{\fontsize{\fsize}{#1\fsize}\selectfont}%
  \ifx\svgwidth\undefined%
    \setlength{\unitlength}{246.26899719bp}%
    \ifx\svgscale\undefined%
      \relax%
    \else%
      \setlength{\unitlength}{\unitlength * \real{\svgscale}}%
    \fi%
  \else%
    \setlength{\unitlength}{\svgwidth}%
  \fi%
  \global\let\svgwidth\undefined%
  \global\let\svgscale\undefined%
  \makeatother%
  \begin{picture}(1,0.64477057)%
    \lineheight{1}%
    \setlength\tabcolsep{0pt}%
    \put(0,0){\includegraphics[width=\unitlength,page=1]{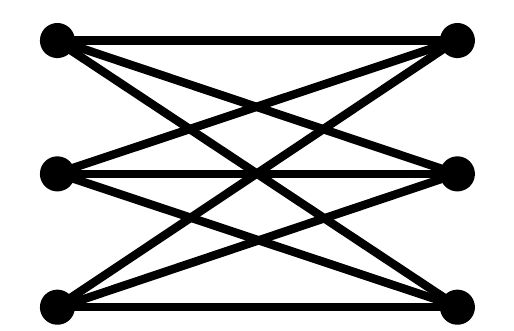}}%
    \put(0.00980343,0.11520737){\color[rgb]{0,0,0}\makebox(0,0)[lt]{\lineheight{1.25}\smash{\begin{tabular}[t]{l}$s_1$\end{tabular}}}}%
    \put(0.00980306,0.37508986){\color[rgb]{0,0,0}\makebox(0,0)[lt]{\lineheight{1.25}\smash{\begin{tabular}[t]{l}$s_2$\end{tabular}}}}%
    \put(0.00980306,0.63496828){\color[rgb]{0,0,0}\makebox(0,0)[lt]{\lineheight{1.25}\smash{\begin{tabular}[t]{l}$s_3$\end{tabular}}}}%
    \put(0.87952308,0.63496422){\color[rgb]{0,0,0}\makebox(0,0)[lt]{\lineheight{1.25}\smash{\begin{tabular}[t]{l}$t_3$\end{tabular}}}}%
    \put(0.87952302,0.37508986){\color[rgb]{0,0,0}\makebox(0,0)[lt]{\lineheight{1.25}\smash{\begin{tabular}[t]{l}$t_2$\end{tabular}}}}%
    \put(0.87952302,0.11521143){\color[rgb]{0,0,0}\makebox(0,0)[lt]{\lineheight{1.25}\smash{\begin{tabular}[t]{l}$t_1$\end{tabular}}}}%
  \end{picture}%
\endgroup%

		\caption{A non-planar balanced DAG whose unit flow polytope admits no lattice triangulation.}
		\label{fig:no_triangulation}
	\end{figure}
	Perfect matchings of $\G$ are of the form $A=\{\gamma_{1,a(1)},\gamma_{2,a(2)},\gamma_{3,a(3)}\}$, where $a$ is a permutation of $[3]$. 
	For any perfect matching $A$ of $\G$, there is a unit flow $\J(A)$ of $\G$ labelling the three edges of $A$ with 1 and all other edges of $\G$ with 0.
	Since all edges of $\G$ are incident to a source and sink, no unit flow may label an edge with a number greater than 1. It follows that the lattice points of $\F_\G(1)$ are precisely those flows of the form $\J(A)$ for a perfect matching $A$.
	There are six permutations of $[3]$, hence six perfect matchings of $\G$ and six lattice points of $\F_\G(1)$.

	We will show that it is impossible for a pairwise compatibility condition on the integer points of $\F_\G(1)$ to induce a triangulation of $\F_\G(1)$.
	Suppose for contradiction that there does exist such a pairwise compatibility condition indexing a triangulation $\mathcal T$ of $\F_\G(1)$.
	Say that a \emph{clique} is a collection of pairwise compatible perfect matchings, so that the cliques correspond to the simplices of $\mathcal T$.

	Suppose $A=\{\gamma_{1,a(1)},\gamma_{2,a(2)},\gamma_{3,a(3)}\}$ and $B=\{\gamma_{1,b(1)},\gamma_{2,b(2)},\gamma_{3,b(3)}\}$ are perfect matchings of $\G$ (so that $a$ and $b$ are permutations of $[3]$). We claim that if $C$ is a perfect matching using only edges of $A\cup B$, then $C=A$ or $C=B$.

	The claim is immediate if $A=B$, so suppose $A$ and $B$ are distinct. By symmetry, it suffices to consider the following cases:
	\begin{enumerate}
		\item Suppose $A\cup B$ consists of 6 distinct edges. Take a perfect matching $C$ using only edges of $A\cup B$. Then $C$ must contain either $\gamma_{1,a(1)}$ or $\gamma_{1,b(1)}$; by switching $A$ and $B$ if necessary, suppose without loss of generality that $C$ contains $\gamma_{1,a(1)}$.
			Then $C$ cannot contain the edge $\gamma_{1,b(1)}$ because $C$ already contains an edge starting at $s_1$. Since $C$ must contain an edge ending at $b(1)$, it must be that $C$ contains $\gamma_{a^{-1}(b(1)),b(1)}$ (this is distinct from the edge $\gamma_{1,a(1)}$ because $a(1)\neq b(1)$ by hypothesis). The only perfect matching containing both $\gamma_{1,a(1)}$ and $\gamma_{a^{-1}(b(1)),b(1)}$ is $A$, so we have $C=A$, completing the proof in this case.
		\item Suppose there exists $j\in[3]$ that $a(j)=b(j)$.
			For notational convenience, we take $j=3$; the cases $j=1$ and $j=2$ flow symmetrically.
			Since $A$ and $B$ are distinct, $a(1)\neq b(1)$ and $a(2)\neq b(2)$. 
			It follows that $b(1)=a(2)$ and $a(2)=b(1)$.
			Take a perfect matching $C$ using only edges of $A\cup B$ -- then $\gamma_{3,a(3)}=\gamma_{3,b(3)}\in C$.  The matching $C$ must contain either $\gamma_{1,a(1)}$ or $\gamma_{1,b(1)}$; switching $A$ and $B$ if necessary, suppose without loss of generality that $\gamma_{1,a(1)}\in C$.
			Then $C$ still must contain an arrow ending at $a(2)=b(1)$, and it cannot be $\gamma_{1,b(1)}$ since $\gamma_{1,a(1)}\in C$. Then $C$ must contain $\gamma_{2,a(2)}$, completing the proof that $C=A$.
	\end{enumerate}

	We have shown that the only perfect matchings of $\G$ using only edges of $A\cup B$ are $A$ and $B$ themselves. It follows that the only lattice points of $\F_\G(1)$ which give flow of 0 to all edges outside of $A\cup B$ are $\J(A)$ and $\J(B)$. Then
	the only expression of the flow $F:=\frac{\J(A)+\J(B)}{2}$ as a positive linear combination of indicator vectors of perfect matchings is the sum $\frac{\J(A)+\J(B)}{2}$ itself. Because $\mathcal T$ is a triangulation of $\F_\G(1)$, the unit flow $F$ must appear in a simplex of $\mathcal T$ whose vertices are indicator vectors of perfect matchings; this simplex must then contain both $\J(A)$ and $\J(B)$. This shows that $A$ and $B$ must be compatible.

	Since $A$ and $B$ were chosen arbitrarily, we have shown that any two distinct perfect matchings of $\G$ are compatible. It follows there is only one maximal clique, and that this clique consists of all six perfect matchings of $\G$. This shows that $\F_\G(1)$ is a five-dimensional simplex. On the other hand, we may show that $\F_\G(1)$ is not a five-dimensional simplex by finding an affine dependence of its vertices: one may verify that
	\begin{align*}
		&\J(\{\gamma_{1,1},\gamma_{2,2},\gamma_{3,3}\})+
		\J(\{\gamma_{1,2},\gamma_{2,3},\gamma_{3,1}\})+
		\J(\{\gamma_{1,3},\gamma_{2,1},\gamma_{3,2}\}) \\=
		&\J(\{\gamma_{1,1},\gamma_{2,3},\gamma_{3,2}\})+
		\J(\{\gamma_{1,2},\gamma_{2,1},\gamma_{3,3}\})+
		\J(\{\gamma_{1,3},\gamma_{2,2},\gamma_{3,1}\})
	\end{align*}
	by checking that both flows give every edge of $\G$ a flow of 1.
	This is a contradiction.
	This completes the proof that there there is no lattice triangulation $\mathcal T$ of the flow polytope $\F_\G(1)$ induced by a pairwise compatibility condition on its integer points.
\end{example}

\subsection{Future directions}

We conclude the article with some possible directions of future research.

In light of Example~\ref{ex:k33}, one cannot hope to define a notion of framing triangulations for general flow polytopes without significantly changing methodology.
On the other hand, it is still an interesting question to consider for which classes of balanced DAGs one can extend variants of the theory of framing triangulations.
\begin{question}
	Which classes of balanced DAGs admit some theory of framing triangulations and framing posets?
\end{question}

While framing posets of strongly planar DAGs may fail to be lattices, it appears through our observation and testing that these posets display a number of properties of framing lattices.

For example, recall that a lattice is a poset where any two elements $x$ and $y$ have a unique least upper bound and a unique greatest lower bound. While elements $x$ and $y$ of a framing poset $\mathcal L$ of a strongly planar DAG may fail to have an upper bound in $\mathcal L$, it appears that $x$ and $y$ must have \emph{at most one} least upper bound and \emph{at most one} greatest lower bound. In other words, ``existence of meets'' may fail in a framing poset but ``uniqueness of meets'' seems to hold.

Expanding on this, framing lattices $\mathcal L$ are known to be \emph{polygonal}~\cite[Theorem 1.3.4]{vBC} where all polygons are squares, pentagons, and hexagons. To be precise, if $x,y\in\mathcal L$ cover a common element $z$ (resp. are covered by a common element $z$), then the interval $[z,x\lor y]$ (resp. $[x\land y,z]$) is a square, pentagon, or hexagon.
It appears that framing posets of strongly planar DAGs display some version of polygonality for posets. In particular, if $x$ and $y$ cover a common element $z$ and there exists a least upper bound $x\lor y$ for $x$ and $y$ then the interval $[z,x\lor y]$ is a square or hexagon (and the dual result for lower covers). In other words, ``existence of polygons'' may fail in a framing poset but ``uniqueness of polygons'' seems to hold.
\begin{conjecture}
	Framing posets of strongly planar DAGs display some polygonality properties into squares and hexagons.
\end{conjecture}

Finally, we briefly recall a connection between strongly planar DAGs in the one-source-one-sink case and order polytopes developed by M\'esz\'aros, Morales, and Striker~\cite{MMS}; we refer to their work for more information on this subject and the definitions of strongly planar posets and their order polytopes. A strongly planar DAG $\G$ with one source and one sink is associated to a \emph{strongly planar} poset $P_\G$. There exists an integral equivalence $\phi:\F_1(\G)\to\hat{\mathcal O}(P_\G)$ induced by the planar embeddings, where $\hat{\mathcal O}(P_\G)$ is the order polytope of $P_\G$. Moreover, this integral equivalence $\phi$  maps the framing triangulation of $\F_1(\G)$ to the canonical triangulation of $\hat{\mathcal O}(P_\G)$~\cite[Theorem 1.3]{MMS}.
In this way, the theory of framing triangulations of strongly planar DAGs with one source and one sink is the same as the theory of canonical triangulations of strongly planar posets.
It would be interesting to expand this correspondence to our new strongly planar setting allowing multiple sources and sinks.
\begin{question}
	Is there a generalization of strongly planar posets whose theory of canonical triangulations is the same as the theory of framing triangulations of strongly planar DAGs?
\end{question}

\bibliographystyle{alphaurl}
\bibliography{biblio} 

\end{document}